\documentclass[11pt]{amsart}
\usepackage{amssymb,amsmath,amsthm,amsfonts,mathrsfs}
\usepackage[hmargin=3cm,vmargin=3.5cm]{geometry}
\usepackage[dvipsnames,table,xcdraw]{xcolor}
\usepackage[colorlinks = true,
            linkcolor = ForestGreen,
            urlcolor  = Fuchsia,
            citecolor = blue,
            anchorcolor = blue]{hyperref}

\usepackage{stackengine}

\usepackage[all,2cell]{xy} \UseAllTwocells 
\usepackage{tikz-cd}
\usepackage{quiver} 

\usepackage{color}
\usepackage{graphicx,psfrag}
\usepackage{amscd}  
\usepackage{stmaryrd}  
\usepackage[all,2cell]{xy} \UseAllTwocells \SilentMatrices
\usepackage{tikz}
\usepackage[dvips]{epsfig}
\usepackage{comment}
\usepackage{kbordermatrix}
\usepackage{cancel,scalefnt}
\usepackage{stackengine}
\usepackage{soul}

\usepackage{array}
\newcolumntype{L}{>{$}l<{$}} 

\theoremstyle{definition}
\newtheorem{theorem}{Theorem}[section]
\newtheorem{lemma}[theorem]{Lemma}
\newtheorem{remark}[theorem]{Remark}

\newtheorem{proposition}[theorem]{Proposition}

\newtheorem{corollary}[theorem]{Corollary}
\newtheorem{example}[theorem]{Example}

\catcode`~=11 
\newcommand{\urltilde}{\kern -.15em\lower .7ex\hbox{~}\kern .04em}  
\catcode`~=13 

\usetikzlibrary{arrows,automata}
\usetikzlibrary{decorations.markings}
\usetikzlibrary{decorations.pathmorphing,shapes,decorations.text,shapes.geometric}
\usepgflibrary {shadings}
\usepackage{multirow}

\newcounter{sarrow}

\usepackage[colorinlistoftodos]{todonotes}

\definecolor{darkred}{rgb}{0.7,0,0} 
\newcommand{\defn}[1]{{\color{darkred}\emph{#1}}} 

\title[Multisymmetric functions on eventually constant cyclic graphs]{Multisymmetric functions on eventually constant cyclic graphs}

\author[Radford Green]{Radford Green}
\address{Department of Mathematics, Johns Hopkins University, Baltimore, MD 21218, USA}
\email{\href{mailto:rgreen87@jh.edu}{rgreen87@jh.edu}}

\author[Cornell Holmes]{Cornell Holmes}
\address{Department of Mathematics, Johns Hopkins University, Baltimore, MD 21218, USA}
\email{\href{mailto:cholme21@jhu.edu}{cholme21@jhu.edu}}

\author[Mee Seong Im]{Mee Seong Im}
\address{Department of Mathematics, Johns Hopkins University, Baltimore, MD 21218, USA}
\email{\href{mailto:meeseong@jhu.edu}{meeseong@jhu.edu}}

\makeatletter
\@namedef{subjclassname@2020}{%
  \textup{2020} Mathematics Subject Classification}


\providecommand{\keywords}[1]{\textbf{\textit{Key words and phrases.}} #1}


\date{August 7, 2026}
\begin{document}

\def\mfb{\mathfrak{b}}

\def\wt{\mathsf{wt}}

 \def\A{\mathbb{A}}

\def\E{\mathsf E}
\def\F{\mathbb{F}}
\def\I{\mathsf{I}}
\def\R{\mathbb R}
\def\Q{\mathbb Q}
\def\Z{\mathbb Z}
\def\N{\mathbb N}
\def\C{\mathbb C}
\def\S{\mathbb S}
\def\Kr{\mathsf{Kr}}
\def\sgn{\mathsf{sgn}}
\def\Sym{\mathsf{Sym}}
\def\Leinster{\mathsf{Leinster}}
\def\Lin{\mathsf{Lin}}
\def\Nil{\mathsf{Nil}}
\def\lrarrow{\mathrel{\stackanchor[0.10pt]{$\leftarrow$}{$\rightarrow$}}}

\def\diag{\mathsf{diag}}
\def\EC{\mathsf{EC}}

\def\lab{\mathsf{label}}

\def\indeg{\mathsf{indeg}}
\def\outdeg{\mathsf{outdeg}}

\def\orig{\mathsf{orig}}
\def\forest{\mathsf{forest}}

\def\SS{\mathbb S}
\def\GL{\mathsf{GL}}
\def\Graph{\mathsf{Graph}}
\def\top{\mathsf{top}}

\def\for{\mathsf{for}}
\def\Hom{\mathsf{Hom}}
\def\End{\mathsf{End}}

\def\Der{\mathsf{Der}}
\def\Pol{\mathsf{Pol}}

\newcommand{\dmod}{\mathsf{-mod}}
\newcommand{\comp}{\mathrm{comp}} 
\newcommand{\col}{\mathrm{col}}
\newcommand{\adm}{\mathrm{adm}}  
\newcommand{\Ob}{\mathrm{Ob}}
\newcommand{\Cob}{\mathsf{Cob}}
\newcommand{\UCob}{\mathsf{UCob}}
\newcommand{\COB}{\mathcal{COB}}
\newcommand{\ECob}{\mathsf{ECob}}
\newcommand{\id}{\mathsf{id}}
\newcommand{\undM}{\underline{M}}
\newcommand{\im}{\mathsf{im}}
\newcommand{\coker}{\mathsf{coker}}
\newcommand{\Aut}{\mathsf{Aut}}
\newcommand{\tripod}{\mathsf{Td}}
\newcommand{\BBC}{\mathbb{B}(\mathcal{C})}
\newcommand{\Pmod}{\mathrm{pmod}}
\newcommand{\gammaoneR}{\gamma_{1,R}}  
\newcommand{\gammaoneRbar}{\overline{\gamma}_{1,R}} 
\newcommand{\gammaoneRprime}
{\gamma'_{1,R}}
\newcommand{\gammaoneRbarprime}
{\overline{\gamma}'_{1,R}}
\newcommand{\qbinom}[3]{\genfrac{[}{]}{0pt}{}{#1}{#2}_{#3}}

\def\l{\lbrace}
\def\r{\rbrace}
\def\o{\otimes}
\def\lra{\longrightarrow}
\def\ed{\mathsf{ed}}
\def\Ext{\mathsf{Ext}}
\def\ker{\mathsf{ker}}
\def\mf{\mathfrak} 
\def\mcC{\mathcal{C}}
\def\mcS{\mathcal{S}}  
\def\mcQC{\mathcal{QC}}
\def\mcA{\mathcal{A}}
\def\mcF{\mathcal{F}}
\def\mcE{\mathcal{E}}
\def\Fr{\mathsf{Fr}}  

\def\bbn{\mathbb{B}^n}
\def\ovb{\overline{b}}
\def\tr{{\sf tr}} 
\def\det{{\sf det }} 
\def\one{\mathbf{1}}   
\def\kk{\mathbf{k}}  
\def\gdim{\mathsf{gdim}}  
\def\rk{\mathsf{rk}}
\def\IET{\mathsf{IET}}
\def\SAF{\mathsf{SAF}}

\newcommand{\indexw}{\R_{>0}} 

\newcommand{\brak}[1]{\ensuremath{\left\langle #1\right\rangle}}
\newcommand{\oplusop}[1]{{\mathop{\oplus}\limits_{#1}}}
\newcommand{\addfigure}{\vspace{0.1in} \begin{center} {\color{red} ADD FIGURE} \end{center} \vspace{0.1in} }
\newcommand{\add}[1]{\vspace{0.1in} \begin{center} {\color{red} ADD FIGURE #1} \end{center} \vspace{0.1in} }
\newcommand{\vspin}{\vspace{0.1in} }

\newcommand\circled[1]{\tikz[baseline=(char.base)]{\node[shape=circle,draw,inner sep=1pt] (char) {${#1}$};}} 

\let\oldemptyset\emptyset
\let\emptyset\varnothing

\let\oldtocsection=\tocsection
\let\oldtocsubsection=\tocsubsection
\renewcommand{\tocsection}[2]{\hspace{0em}\oldtocsection{#1}{#2}}
\renewcommand{\tocsubsection}[2]{\hspace{1em}\oldtocsubsection{#1}{#2}}

\renewcommand{\kbldelim}{(}
\renewcommand{\kbrdelim}{)}

\def\MK#1{{\color{red}[MK: #1]}}
\def\bfred#1{{\color{red}#1}}


\begin{abstract}
The study of spanning trees and related structures is central in graph theory, closely connected to  understanding functions between finite sets. This paper generalizes the established relationship between rooted trees and eventually constant endomorphisms to a wider context including $k$-tuples of functions among $k$ disjoint vertex sets. We derive a weighted count of eventually constant $k$-tuples, which are characterized by their stabilization to constancy upon iterated composition. This construction is the set-theoretic analogue of the nilpotent cone and offers new insight into the combinatorial structure of cyclic digraphs. By identifying these $k$-tuples with their induced digraphs, we construct explicit formulas for their generating polynomials and analyze the cardinality of the set of eventually constant $k$-tuples. These polynomials are multisymmetric in $k$ sets of variables and can be re-expressed as the character of a representation of the product of general linear groups. We extend the ideas to the more general structures of eventually $N$-cyclic and $\lambda$-cyclic $k$-tuples, which we define and provide similar theorems for their generating functions and cardinality.
\end{abstract}

\maketitle
\tableofcontents

%
%

\section{Introduction}
\label{section_intro}

The enumeration of spanning trees is fundamental in graph theory, and matrix-tree theorems~\cite{CK78} yield weighted enumerations of spanning trees as suitable determinants. For example, the number of labeled spanning trees of the complete bipartite graph on $m$ and $n$ vertices is $m^{n-1}n^{m-1}$~\cite{FS58_eng}. Various reductions of matrix-tree calculations to simpler graphs were later developed by~\cite{PakPostnikov94}.

We can understand endomorphisms through the study and enumeration of trees and permutations. In the simplest case, highlighted in~\cite{Lei21}, the rooted trees of an $n$-element set $[n]$ correspond to the \defn{eventually constant} endomorphisms, maps $f:[n]\to [n]$ whose iterated compositions $f^m=f\circ\cdots\circ f$ eventually stabilize to a constant. Cayley's formula, a classical result that has been extensively generalized by counting edges with corresponding formal variables~\cite{Cayley09, Hai10, Sta24, Die25, Wes96, Bol98}, states that there are $n^{n-1}$ rooted trees. Thus, there are $n^{n-1}$ eventually constant endomorphisms, and an endomorphism is eventually constant with probability
\begin{equation}\label{k=1_EC_prob_intro}
    \frac{n^{n-1}}{n^n}=\frac{1}{n}.
\end{equation}
More broadly any endomorphism $f:[n]\to [n]$ decomposes into a permutation and a forest~\cite{FS09}; iterated images $f^m([n])\subseteq[n]$ stabilize to a subset $U\subseteq [n]$ on which $f$ acts as a permutation. On the remaining vertices, iterating $f$ leads to the repeated vertices in $U$, so the action of $f$ on $[n]\setminus U$ induces an in-forest rooted at $U$.

The relation between spanning trees and eventually constant maps has been explored from various perspectives by~\cite{Lei21, FH58, GKP89, Sta24}. The authors in~\cite{CIKLR25, CILR25} generalized this framework from a single map on one set to pairs of linear maps between two sets. 
In this paper, we extend the notion of eventual constancy to $k$-tuples of maps among $k$ distinct sets and derive a weighted enumeration of this class. Then, we progressively expand this class and analyze each generalization in the same manner. These constructions act as a set-theoretic analogue of the nilpotent cone, studied in combinatorics, geometry, topology, and geometric representation theory. We describe the construction precisely as follows.

Fix a positive integer $k\in\mathbb N$ and $k$ sets $X_1, \ldots, X_k$ of cardinality $n_1, \,\ldots, \,n_k$, respectively. Then, a $k$-tuple of functions $(f_1,\ldots,f_k)$ with the following chain
\[
X_1 \xrightarrow{f_1}X_2\xrightarrow{f_2}\ldots \xrightarrow{f_{k-2}}X_{k-1}\xrightarrow{f_{k-1}}X_k\xrightarrow{f_k}X_1
\]
of maps is \defn{eventually constant} if the composition $f_k \cdots f_1: X_1 \to X_1$ eventually stabilizes to a constant map, or equivalently, if $f_k \cdots f_1$ is eventually constant in the traditional sense.

We assign each $k$-tuple $f:=(f_1,\ldots,f_k)$ the formal variable product of its images, its polynomial \defn{weight}
\begin{equation}
\label{k-tuple weight definition}
   w(f) :=\prod_{i=1}^k \prod_{v\in X_i} x_{f_i(v)},
\end{equation}
and denote the \defn{$j^{th}$ elementary symmetric polynomial} in vertices $v\in U$ by 
\begin{equation}\label{Elementary symmetric polynomials}
    e_j(U):=\sum_{\substack{S\subseteq U\\ |S|=j}}\prod_{v\in S} x_{v}.
\end{equation}
Let $\EC$ denote the set of eventually constant $k$-tuples, and define its generating function, $P(\mathbf x) :=\sum_{f\in\EC} w(f)$.

We derive a multisymmetric weighted enumeration in Theorem~\ref{EC weight theorem}:
\begin{theorem}
\label{Theorem statement intro}
The weighted enumeration of all eventually constant $k$-tuples on $X_1, \ldots, X_k$ is given by
\begin{equation} \label{eq_P_product}
P(\mathbf x) 
= 
\prod_{i=1}^ke_1(X_i)^{n_{i-1}-2} 
\left(
\prod_{j=1}^ke_1(X_j)^{2} - 2^k\prod_{j=1}^ke_2(X_j)
\right),
\end{equation}
where the indices are understood modulo $k$, so that $n_{1-1}=n_k$. 
\end{theorem}

Throughout this paper, indices are understood modulo $k$, so that $n_{1-1}=n_k$. We may re-express \eqref{eq_P_product} as a virtual character of representations of $\prod_{i=1}^k \GL(n_i)$: 
\begin{equation}\label{eq_P_char}
P(\mathbf x)=
\chi\!\left(
\bigotimes_{i=1}^k V_i^{\otimes n_{i-1}}
-
2^k
\bigotimes_{i=1}^k
\left(
V_i^{\otimes (n_{i-1}-2)}\otimes \Lambda^2V_i
\right)
\right),
\end{equation}
where $V_i\cong \mathbb{C}^{n_i}$ denotes the fundamental representation
of $\GL(n_i)$; see Corollary~\ref{rep_theory_formulation}.

To derive the weighted enumeration \eqref{eq_P_product}, we identify each $k$-tuple with its directed graph.
Each digraph decomposes into a cyclic component and a forest directed towards its cyclic component. We partition $\EC$ by each cyclic vertex set and construct a weight-preserving bijection between each equivalence class and its corresponding pairs of cyclic and forest digraphs. The cyclic components are simple to weigh, since eventually constant $k$-tuples collapse to a single cycle on each fixed vertex set. Then, we apply the directed matrix-tree theorem~\cite{CK78} to weigh the forests. Finally, the uniformity of the distribution of cyclic vertex sets, cycle weights, and forest weights allows us to reduce the resulting sum over class weights into the elementary symmetric form \eqref{eq_P_product}.

Setting the algebraic weights $x_u$ to $1$ in \eqref{eq_P_product}, we recover the cardinality of the eventually constant $k$-tuples in Proposition \ref{EC Cardinality}:
\begin{proposition}
\label{Proposition statement intro}
The cardinality of eventually constant $k$-tuples on $X_1, \ldots, X_k$ is
\begin{equation} \label{EC_cardinality_intro}
\#\EC 
= \prod_{i=1}^kn_i^{n_{i-1}-1}\cdot
\left(
\prod_{j=1}^kn_j - \prod_{j=1}^k(n_j-1)
\right).
\end{equation}
\end{proposition}
This generalizes~\cite[Theorem 2.3]{CIKLR25}.

We generalize \eqref{k=1_EC_prob_intro}, the probability that a single endomorphism is eventually constant. Using \eqref{EC_cardinality_intro}, we find the probability $p_k(n)$ that a $k$-tuple of functions between $k$ equal sets $[n],\ldots,[n]$ is eventually constant in Corollary~\ref{Uniform EC asymptotic}:
\begin{equation}\label{Uniform EC asymptotic intro}
p_k(n)= \frac{k}{n} +\mathcal O\left(\frac{1}{n^2}\right).
\end{equation}

Next, we fix a positive integer $N\in\mathbb N$ and generalize from constant maps to cycles. A $k$-tuple $(f_1,\ldots,f_k)$ is \defn{eventually $N$-cyclic} if the composition $f_k\cdots f_1:X_1\to X_1$ eventually stabilizes to an $N$-cycle on $X_1$. Therefore, an eventually constant $k$-tuple is eventually $1$-cyclic. Let $\EC(N)$ denote the set of eventually $N$-cyclic $k$-tuples, and define the corresponding generating function, $P_N(\mathbf x):=\sum_{f\in\EC(N)}w(f)$.

To compute $P_N(\mathbf x)$, we extend the cyclic and forest decomposition and partition techniques used in the previous weighted enumeration. Since the $N$-cyclic vertex sets grow as $N$ does, they allow for distinct cycle structures. After calculating the weight of the admissible cycles, we again apply the directed matrix-tree theorem, and then reduce the resulting sum using the uniformity of the $N$-cyclic vertex sets.

This yields the multisymmetric weighted enumeration in Theorem~\ref{Weight of eventually cyclic functions}:
 \begin{theorem}
\label{thm_wt_eventally_cyclic_fns}
The weighted enumeration of all eventually $N$-cyclic $k$-tuples is given by 
\begin{equation}\label{eqn_wt_eventally_cyclic_fns}
    P_N(\mathbf{x}) = \frac{(N!)^{k}}{N} 
    \prod_{i=1}^k e_1(X_i)^{n_{i-1}-(N+1)}
    \left(
    \prod_{j=1}^k e_N(X_j)e_1(X_j) -(N+1)^k\prod_{j=1}^k e_{N+1}(X_j)
    \right).
\end{equation}
\end{theorem}
As with the polynomial $P(\mathbf x)$, we may re-express $P_N(\mathbf x)$ via the characters of representations of $\prod_{i=1}^k\GL(n_i)$ as 
\[
\frac{(N!)^{k}}{N}
\;
\chi\!\left(
\bigotimes_{i=1}^k
\left(
V_i^{\otimes (n_{i-1}-N)} \otimes \Lambda^N V_i
\right)
-
(N+1)^k
\bigotimes_{i=1}^k
\left(
V_i^{\otimes (n_{i-1}-(N+1))} \otimes \Lambda^{N+1} V_i
\right)
\right),
\]
where $V_i \cong \mathbb{C}^{n_i}$ is the fundamental representation of $\GL(n_i)$.

We use \eqref{eqn_wt_eventally_cyclic_fns} to derive the cardinality of the eventually $N$-cyclic $k$-tuples in Proposition \ref{prop_card_EC(N)}:
\begin{proposition}
The cardinality of eventually $N$-cyclic $k$-tuples on $X_1,\dots,X_k$ is
\begin{equation}\label{card_ECN_intro}
\#\EC(N)
=
\frac{1}{N}
\prod_{i=1}^k n_i^{\,n_{i-1}-(N+1)}\frac{n_i!}{(n_i-N)!}
\left(
\prod_{j=1}^k n_j - \prod_{j=1}^k (n_j - N)
\right).
\end{equation}
\end{proposition}
Surprisingly, the probability $p_{k,N}(n)$ that a $k$-tuple of functions between $k$ equal sets $[n],\ldots,[n]$ is eventually $N$-cyclic for fixed $N$ shares the asymptotic expansion with $p_k(n)$ in \eqref{Uniform EC asymptotic intro}. In Corollary \ref{Uniform ECN asymptotic}, 
we find
\begin{equation}\label{limit_ECN_prob_intro}
    p_{k,N}(n) = \frac{k}{n} + \mathcal O\left(\frac{1}{n^2}\right).
\end{equation}
However, the cycle size parameter $N$ does affect the limit if we scale both $n$ and $k:=cn$ for a fixed positive constant $c\in\mathbb N$. In Corollary \ref{Proportional ECN asymptotic}, we derive the limit of the probability $p_{cn,N}(n)$ that a $cn$-tuple is eventually $N$-cyclic as 
\begin{equation}\label{limit_p_cnN(n)_limit_intro}
    \lim_{n \to \infty} p_{cn,N}(n) = \frac{1}{N} e^{-c \binom{N}{2}} \left( 1 - e^{-cN} \right). 
\end{equation}

We extend the concept once more from $N$-cycles to cycle partitions. Fix a positive integer $N\in\mathbb N$ and an integer partition $\lambda\vdash N$. For each number $l\leq N$, let $m_l$ denote the multiplicity of $l$ in $\lambda$. A $k$-tuple $(f_1,\ldots,f_k)$ is \defn{eventually $\lambda$-cyclic} if the composition $f_k\cdots f_1:X_1\to X_1$ eventually stabilizes to a permutation of \defn{cycle type $\lambda$} on a subset of cardinality $N$. Let $\EC(\lambda)$ be the set of eventually $\lambda$-cyclic $k$-tuples, and define the generating function, $P_\lambda(\mathbf x):=\sum_{f\in\EC(\lambda)}w(f)=w(\EC(\lambda))$.

To compute $P_\lambda(\mathbf x)$, we again partition $\EC(\lambda)$ by cyclic vertex sets and weigh each class by its cyclic and forest components. The cyclic vertex sets are as in the eventually $N$-cyclic case, so the weight of each cyclic component is preserved, as is the weight of the forests. All that changes is the cardinality of admissible cyclic structures, which are proportional to the number of permutations of cycle type $\lambda$~\cite{Sta11}.

Thus, we realize $P_\lambda(\mathbf x)$ and the cardinality $\#\EC(\lambda)$ as a ratio of $P_N(\mathbf x)$ and $\#\EC(N)$:
\begin{equation*}
    P_\lambda(\mathbf x)=\frac{N}{\prod_{l=1}^N l^{m_l}m_l!} P_N(\mathbf x),\quad \#\EC(\lambda)=\frac{N}{\prod_{l=1}^N l^{m_l}m_l!} \#\EC(N).
\end{equation*}
See~\eqref{Plambda pn ratio} and Proposition~\ref{lambda_cyclic_generating_function}.

Note that the partition generalization classifies all $k$-tuples. The composition $f_k\cdots f_1$ of any $k$-tuple eventually stabilizes to a subset $U_1\subseteq X_1$ on which it acts as a permutation, and thus admits a cycle type.

We study multisymmetric polynomials on a generalized family of oriented graphs in~\cite{GHI26_set_quiver}, but we do not explore virtual characters associated to them.

Section~\ref{section_background} provides the necessary graph-theoretic background. In Section~\ref{The weight of eventually constant functions}, we give canonical multisymmetric polynomials associated to $k$-tuple of cyclic functions that are eventually constant. In Section~\ref{The cardinality of eventually constant functions}, we prove that evaluating one into the multivariables gives the enumeration of the number of eventually constant functions for a $k$-tuple. In Section~\ref{Specializations: k=2 and k=3}, we specialize to the cases of eventually constant pairs and eventually constant triples and give multisymmetric functions and the cardinalities associated to them. We also rewrite the multisymmetric functions as virtual characters of representations.

  Moving to eventually cyclic functions in Section~\ref{Eventually cyclic functions}, we extend the eventually constant framework, providing multisymmetric polynomials associated to eventually cyclic functions in Section~\ref{subsection_wt_eventually_cyclic_fns} (see Section~\ref{Eventually cyclic functions} for the definition of eventually $N$-cyclic $k$-tuple). In Section~\ref{Enumeration_eventually_cyclic}, we give the enumeration of eventually cyclic functions. We then specialize to eventually cyclic pairs and triples in Section~\ref{Eventually_cyclic_pairs_and_triples}.

  In Section~\ref{section_lambda_cyclic_functions}, we define eventually $\lambda$-cyclic $k$-tuple. We then give multisymmetric polynomials associated to $\lambda$ in Section~\ref{lambda_cyclic_functions}, and give an example of a $\lambda$-cyclic triple and discuss its enumeration using a counting argument.

\subsection*{Acknowledgments}
The authors are grateful to Mikhail Khovanov for insightful discussions. The authors thank the anonymous referee for the helpful comments.
The authors would like to thank the Department of Mathematics and Erin Kathleen Rowe and Jasmine SharDae' Jenkins at the Dean's Office of Krieger School of Arts $\&$ Sciences at Johns Hopkins University for the opportunity to conduct research and for their support. 
M.S. Im would like to thank the Simons Foundation in New York, NY for a dynamic and collaborative working atmosphere. 
The authors were partially supported by Simons Collaboration Award 994328. C. Holmes was also partially supported by NSF grant DMS-2428878.

\section{Background}
\label{section_background}
\subsection{Introduction to graphs}\label{Graph-theoretic framework}
We recall the necessary graph-theoretic terminologies.

We consider finite directed graphs (digraphs) that allow for loops but not multiple edges between vertices. A \defn{{digraph}} on a vertex set $U$ is a pair $D:=(U,E_D)$, where $E_D$ is a \defn{(directed) edge set}, a subset of ordered vertex pairs $(u,v)\in U\times U$. The edge $(u,v)$ is directed from $u$ towards $v$. We also denote the vertex set of a digraph $D$ by $V(D):=U$.

Consider digraphs $D$ and $G$. Then, $D$ is a \defn{subgraph} of $G$ if $V(D)\subseteq V(G)$ and $E_D\subseteq E_G$, and $D$ is a \defn{spanning subgraph} of $G$ if $D$ is a subgraph and $V(D)=V(G)$.

The \defn{union} of digraphs $D$ and $G$ is the digraph $(V(D)\cup V(G),E_D\cup E_G)$. When the edge sets $E_D$ and $E_G$ are disjoint, we denote the union graph by $D\oplus G$. Also, for any subset $U\subseteq V(D)$, let $D_U$ denote the spanning digraph formed by removing all outgoing edges from $U$.

A \defn{trail} in $D$ from a vertex $s_1$ to a vertex $s_t$ is a sequence of distinct adjacent edges in $D$ leading from $s_1$ to $s_t$:
\[(s_1,s_2),\,(s_2,s_3),\,\ldots,\,(s_{t-1},s_t).\]

A \defn{cycle} is a nonempty trail in $D$ along vertices $s_1,\ldots,s_t,s_1$ in which $s_1,\ldots,s_{t}$ are distinct, and $D$ is \defn{acyclic} if it does not contain any cycles. A digraph $C$ whose edge set forms a cycle along its vertices is called a \defn{(directed) cycle graph}.

The \defn{out-degree} of a vertex $u\in V(D)$ is the number of outgoing edges $(u,v)$ in $D$; \defn{in-degree} is defined analogously. The digraph $D$ is \defn{functional} if every vertex has out-degree $1$, which holds precisely if $D$ corresponds to the standard set-theoretic graph of an endomorphism $f_D:V(D)\to V(D)$ on $V(D)$. Note that vertices in cycle graphs have in and out-degree $1$, thus are functional. In fact, a nonempty graph $C$ is a cycle graph precisely if it corresponds to a $|V(C)|$-cycle, $f_C:V(C)\to V(C)$.

For each vertex set $U$, we work in the vertex-indexed polynomial ring $\mathbb Z[x_u]_{u\in U}$ and define the \defn{weight} $w(D)$ of any digraph $D$ of $U$ via in-degrees:
\begin{equation}\label{Weight definition}
    w(D):=\prod_{(u,v)\in E_D}x_v.
\end{equation}
Note that if $D$ is functional, then each vertex $u\in U$ has the unique target $f_D(u)$, so the weight becomes the following product over $U$:
\[w(D)=\prod_{u\in U}x_{f_D(u)}.\]
We also extend $w$ additively over any subset $\mathcal S$ of digraphs on $U$:
\[
w(\mathcal{S}):=\sum_{D\in\mathcal{S}}w(D).
\]

A digraph $D$ is an \defn{in-forest} rooted at $R\subseteq V(D)$ if there is at most one trail from each vertex $u\in V(D)$ to a vertex $v\in R$. If $D$ is an in-forest rooted at $R$, outgoing edges from all vertices correspond to distinct nonempty trails leading to $R$, so vertices in $V(D)\setminus R$ have out-degree $1$ and vertices in $R$ have out-degree $0$. Also, observe that all trails terminate in $R$ in at most $|V(D)\setminus R|$ steps, and that $D$ is acyclic.

We decompose functional digraphs into cyclic and forest parts, following \cite{FS09}. Let $D$ be such a digraph with corresponding endomorphism $f_D:V(D)\to V(D)$. For each vertex $u\in V(D)$, the unique trail from $u$ follows the iterated values of $f_D$:
\[u,\,f_D(u),\,f_D(f_D(u)),\,\ldots,\,f^m_D(u),\,\ldots\]
As $D$ is finite, this sequence must begin to repeat. Thus, either $u$ belongs to a cycle if $u$ is repeated, or $u$ leads to a cycle following a trail of vertices that are not repeated.

Let $C(D)$ denote the \defn{cyclic part} of $D$, the subgraph induced on the vertices $u$ that belong to a cycle. The subgraph formed by removing all cycle edges, $D_{V(C(D))}$, is an in-forest rooted at $V(C(D))$. Since the cycle and non-cycle edges partition $E_D$, we have the disjoint decomposition and weight factorization:
\[D=C(D)\oplus D_{V(C(D))},\quad\text{and}\quad w(D)=w(C(D))w( D_{V(C(D))}),
\]
respectively. 

Note that the iterated images $f^m(V(D))$ stabilize to the set of cyclic vertices $V(C(D))$, and that $f_D$ acts as a permutation on this set. We also call the restricted permutation $f_D:V(C(D))\to V(C(D))$ the \defn{cyclic part} of the endomorphism $f_D$.

Since each vertex of $D$ has a single outgoing edge, distinct cycles must have disjoint vertex sets, and the unique trail from every vertex leads to exactly one cycle.

\subsection{Cyclic digraphs}
We now turn to a specific vertex set, the disjoint union of components, $X:=\sqcup_{i=1}^kX_i$, and a specific graph. Let $\Gamma$ be the digraph on $X$ which has an edge from all vertices in each component $X_i$ to all vertices in $X_{i+1}$, with indices modulo $k$. We consider the functional and cyclic subgraphs of $\Gamma$.

A functional subgraph $D$ of $\Gamma$ can be viewed both as its corresponding endomorphism $f_D:V(D)\to V(D)$ and as a $k$-tuple of functions. Setting $V(D)_i:=V(D)\cap X_i$, the map $f_D$ splits into a $k$-tuple via restriction to each component:
\[(f_1,\ldots,f_k),\quad f_i:V(D)_i\to V(D)_{i+1},\quad f_i(u):=f_D(u).\]
Conversely, given any vertex subset $U\subseteq X$ with $U_i:=U\cap X_i$, any $k$-tuple of functions 
\[
(f_1,\ldots,f_k),\quad f_i:U_i\to U_{i+1},
\]
reconstitutes an endomorphism $f_D:U\to U$ given by
\[
f_D(u):=f_i(u),\quad u\in X_i,\quad i=1,\ldots,k,
\]
and thus corresponds to a functional subgraph of $\Gamma$ on $U$. Thus, we recover the weight assignment for spanning $k$-tuples $(f_1,\ldots,f_k)$ in \eqref{k-tuple weight definition}:
\[w(f_1,\ldots,f_k)=\prod_{u\in X}x_{f_D(u)}=\prod_{i=1}^k\prod_{u\in X_i}x_{f_i(u)}.\]

The following structural Lemma \ref{Cycle bijection lemma} characterizes cycles on $\Gamma$. In particular, we see that cycles distribute evenly across each component $X_i$. To study the vertex sets of cycles, we define a \defn{choice set} as a vertex subset $U \subseteq X$ that contains exactly one element from each component $X_i$, denoted $u_i \in U \cap X_i$. More generally, for any positive integer $N\in\mathbb N$, an \defn{$N$-choice set} is a subset $U\subseteq X$ that has $N$ elements in each component, denoted $U_i:=X_i\cap U$. A standard choice set is thus a $1$-choice set.

\begin{lemma}\label{Cycle bijection lemma}
Let $C$ be a functional subgraph of $\Gamma$ with corresponding $k$-tuple $(f_1,\ldots,f_k)$ and endomorphism $f_C:V(C)\to V(C)$. Let $U := V(C)$ and $N :=|U_1|$.

Then $C$ is a cycle graph if and only if each map $f_i:U_i\to U_{i+1}$ is a bijection and the composition $f_k\cdots f_1:U_1\to U_1$ is an $N$-cycle. Consequently, if $C$ is a cycle graph, its vertex set $U$ is an $N$-choice set.
\end{lemma}
\begin{proof}
Fix any vertex $u\in U_1$.

If $C$ is a cycle graph, then $f_C$ is a bijective $U$-cycle, so each component $f_i:U_i\to U_{i+1}$ must be a bijection. Thus, $U$ splits into $k$ components $U_1,\ldots,U_k$ of equal size. Note that $f_C^k$ is equal to the composition $f_k\cdots f_1$ on $U_1$, and $u$ has an orbit of cardinality $|U|=Nk$ under $f_C$. Therefore, $u$ has an orbit of cardinality $\frac{|U|}{k}=N$ under $f_k\cdots f_1$, so $f_k\cdots f_1$ is a $N$-cycle.

Conversely, suppose each map $f_i$ is a bijection and that the composition $f_k\cdots f_1:U_1\to U_1$ acts transitively. The orbit of $u$ under $f_C$ contains the orbit under $f_k\cdots f_1$, which is the entire component $U_1$. Since the elements of each component $U_i$ can be reached by the bijection $f_{i-1}f_{i-2}\cdots f_1:U_1\to U_i$, we see $f_C$ acts transitively on $U$. Thus $C$ is a cycle graph.
\end{proof}

\begin{figure}
    \centering
\begin{tikzpicture}[scale=0.6,decoration={
    markings,
    mark=at position 0.5 with {\arrow{>}}}]

\begin{scope}[shift={(0,0)}]



\node at (0,3.5) {$a$};

\node at (0,0.42) {$b$};

\draw[thick, fill] (0.17,3) arc (0:360: 1.75 mm);

\draw[thick, fill] (0.17,1) arc (0:360: 1.75 mm);

\node at (3.05,3.55) {$a'$};

\node at (3.05,0.45) {$b'$};

\draw[thick, fill] (3.17,3) arc (0:360: 1.75 mm);

\draw[thick, fill] (3.17,1) arc (0:360: 1.75 mm);

\node at (6.15,3.55) {$a''$};

\node at (6.15,0.45) {$b''$};

\draw[thick, fill] (6.17,3) arc (0:360: 1.75 mm);

\draw[thick, fill] (6.17,1) arc (0:360: 1.75 mm);

\draw[thick,postaction={decorate}] (0,3) -- (3,3);

\draw[thick,postaction={decorate}] (0,1) -- (3,1);

\draw[thick,postaction={decorate}] (3,3) -- (6,1);


\draw[thick] (3,1) -- (4.2,1.8);

\draw[thick,postaction={decorate}] (4.8,2.2) -- (6,3);

\draw[thick] (6,3) .. controls (6.85,3.25) and (6.85,4.00) .. (6,4.25);

\draw[thick] (6,1) .. controls (6.85,0.75) and (6.85,0.00) .. (6,-0.25);

\draw[thick] (0,4.25) .. controls (1.5,4.75) and (4.5,4.75) .. (6,4.25);

\draw[thick] (0,-0.25) .. controls (1.5,-0.75) and (4.5,-0.75) .. (6,-0.25);

\draw[thick,postaction={decorate}] (0,4.25) .. controls (-1.25,3.75) and (-1.25,2) .. (0,1);

\draw[thick,postaction={decorate}] (0,-0.25) .. controls (-1.25,0.25) and (-1.25,2) .. (0,3);

\node at (1.85,3.5) {$f_1$};

\node at (1.85,0.5) {$f_1$};

\node at (3.90,3) {$f_2$};

\node at (3.80,1) {$f_2$};

\node at (7, 4.15) {$f_3$};

\node at (6.9,-0.15) {$f_3$};

\end{scope}


\begin{scope}[shift={(12,0)}]



\node at (0,3.5) {$a$};

\node at (0,0.42) {$b$};

\draw[thick, fill] (0.17,3) arc (0:360: 1.75 mm);

\draw[thick, fill] (0.17,1) arc (0:360: 1.75 mm);

\node at (3.05,3.55) {$a'$};

\node at (3.05,0.45) {$b'$};

\draw[thick, fill] (3.17,3) arc (0:360: 1.75 mm);

\draw[thick, fill] (3.17,1) arc (0:360: 1.75 mm);

\node at (6.15,3.55) {$a''$};

\node at (6.15,0.45) {$b''$};

\draw[thick, fill] (6.17,3) arc (0:360: 1.75 mm);

\draw[thick, fill] (6.17,1) arc (0:360: 1.75 mm);

\draw[thick,postaction={decorate}] (0,3) -- (3,3);

\draw[thick,postaction={decorate}] (0,1) -- (3,1);

\draw[thick,postaction={decorate}] (3,3) -- (6,1);


\draw[thick] (3,1) -- (4.2,1.8);

\draw[thick,postaction={decorate}] (4.8,2.2) -- (6,3);

\draw[thick] (6,3) .. controls (6.85,3.25) and (6.85,4.00) .. (6,4.25);

\draw[thick] (6,1) .. controls (6.85,0.75) and (6.85,0.00) .. (6,-0.25);

\draw[thick] (0,4.25) .. controls (1.5,4.75) and (4.5,4.75) .. (6,4.25);

\draw[thick] (0,-0.25) .. controls (1.5,-0.75) and (4.5,-0.75) .. (6,-0.25);

\draw[thick,postaction={decorate}] (0,4.25) .. controls (-0.85,4.05) and (-0.85,3.2) .. (0,3);

\draw[thick,postaction={decorate}] (0,-0.25) .. controls (-0.85,-0.05) and (-0.85,0.8) .. (0,1);

\node at (1.85,3.5) {$f_1$};

\node at (1.85,0.5) {$f_1$};

\node at (3.90,3) {$f_2$};

\node at (3.80,1) {$f_2$};

\node at (7, 4.15) {$f_3$};

\node at (6.9,-0.15) {$f_3$};

\end{scope}

\end{tikzpicture}

    \caption{Left: the composition $f_3f_2f_1$ is not a $2$-cycle on $U_1$. 
    Right: the composition $f_3f_2f_1$ is a $2$-cycle on $U_1$.
    See Lemma~\ref{Cycle bijection lemma} and Example~\ref{example_2_cycle}.}
    \label{fig_00003}
\end{figure}

\begin{example}
\label{example_2_cycle}
Let $k=3$ and $X_1 =\{ a, b \}$, $X_2 = \{a', b' \}$, and 
$X_3 = \{ a'',b'' \}$. 
Let $U= X_1 \sqcup X_2 \sqcup X_3$ be the $2$-choice set.
Define 
\begin{equation*}
\begin{split}
f_1(a) &= a', \quad  f_2(a')= b'', \quad f_3(a'') = b,  \\
f_1(b) &= b', \quad f_2(b') = a'', \quad f_3(b'') = a.
\end{split}
\end{equation*}
See Figure~\ref{fig_00003}, left.
Each $f_i$ is a bijection, satisfying $(f_3f_2f_1)(a) = a$ and $(f_3f_2f_1)(b) = b$. 
Since we have two disjoint, directed cycles of the form  
$a\mapsto a'\mapsto b''\mapsto a$ and 
$b\mapsto b'\mapsto a''\mapsto b$ of length $3$, 
the composition $f_3f_2f_1$ is not a $2$-cycle on $U_1$. By Lemma~\ref{Cycle bijection lemma}, this is not a cycle graph.

If we instead define 
\begin{equation*}
\begin{split}
f_1(a) &= a', \quad  f_2(a')= b'', \quad f_3(a'') = a,  \\
f_1(b) &= b', \quad f_2(b') = a'', \quad f_3(b'') = b,
\end{split}
\end{equation*}
then  $(f_3f_2f_1)(a) = b$ and $(f_3f_2f_1)(b) = a$, giving us a $2$-cycle on $U_1$. See Figure~\ref{fig_00003}, right.
By Lemma~\ref{Cycle bijection lemma}, the functional graph is a one directed cycle of length $6$, where $a\mapsto a'\mapsto b''\mapsto b\mapsto b'\mapsto a''\mapsto a$. Thus, the connectedness of the cycle graph is determined by the cycle structure of the composition $f_3f_2f_1$ of the $k$-tuple.
\end{example}

Therefore, there is a correspondence between the cycles of a spanning functional subgraph $(f_1,\ldots,f_k)$ of $\Gamma$ and the cycles of its composition $f_k\cdots f_1$. The composition $f_k\cdots f_1$ acts as an $N$-cycle on a subset $U_1\subseteq X_1$ of cardinality $N$ if and only if the subgraph $(f_1,\ldots,f_k)$ contains a cycle on an $N$-choice set $U$ extending $U_1$. This correspondence is well-defined since any such subcycle is uniquely determined its intersection $U_1$ with $X_1$. Explicitly, the vertex set $U$ is given by 
\[
U_2:=f_1(U_1),\, U_3:=f_2(U_2),\,\ldots,\,U_k:=f_{k-1}(U_{k-1}).
\]

Next, we define three families of subgraphs of $\Gamma$. Fix any $N$-choice set $U\subseteq X$. Let $\mathcal C_U$ denote the set of cyclic subgraphs of $\Gamma$ on $U$. Restated, $D\in\mathcal C_U$ if $D$ is equal to its cyclic part $C(D)$ and comprised solely of disjoint cycles. Let $\mathcal D_U$ denote the set of spanning functional subgraphs whose cyclic part belongs to $\mathcal C_U$. Finally, let $\mathcal F_U$ denote the set of spanning in-forests of $\Gamma$ rooted at $U$.

\begin{lemma}\label{Cyclic forest weight decomposition}
Fix any vertex subset $U\subseteq X$, any subset $\mathcal C'_U\subseteq\mathcal C_U$, and let $\mathcal D'_U\subseteq\mathcal D_U$ be the corresponding subset of digraphs $D\in\mathcal D_U$ for which the cyclic part $C(D)$ belongs to $\mathcal C'_U$. Then, the map $D\mapsto (C(D),D_U)$ is a bijection from $\mathcal D'_U$ to $\mathcal C'_U\times\mathcal F_U$, 
and
\begin{equation}\label{D'U weight}
    w(\mathcal D'_U)=|\mathcal C'_U|\left(\prod_{u\in U}x_u\right)w(\mathcal F_U).
\end{equation}
\end{lemma}
\begin{proof}
    By the cyclic and forest decomposition in the previous section, $D\mapsto (C(D),D_U)$ is a well defined map from $\mathcal D'_U$ to $\mathcal C'_U\times\mathcal F_U$, and $D = C(D)\oplus D_U$.

    Conversely, fix any $(C,F)\in\mathcal C'_U\times \mathcal F_U$. The disjoint union $C\oplus F$ is a subgraph of $\Gamma$ since $C$ and $F$ are, and it is functional since vertices in $U$ have out-degree $1$ from $C$ and vertices in $X\setminus U$ have out-degree $1$ from $F$. Next, $C(C\oplus F)=C$, since the forest $F$ contributes no cycles, and removing the edges outgoing from $U$ leaves $(C\oplus F)_{U}=F$. Since $C\in\mathcal C'_U$, we have $C\oplus F\in\mathcal D'_U$, and we then see $(C,F)\mapsto C\oplus F$ yields the inverse, since $(C(C\oplus F),(F\oplus C)_U)=(C,F).$

    Since weight factors over disjoint edge unions, 
    we have
    \[
    w(\mathcal D'_U) 
    = \sum_{C\in\mathcal C'_U}\sum_{F\in\mathcal F_U}w(C\oplus F) 
    = \sum_{C\in\mathcal C'_U}\sum_{F\in\mathcal F_U}w(C) w(F) 
    = w(\mathcal C'_U)w(\mathcal F_U).
    \]

    Finally, every vertex in a cyclic graph $C\in\mathcal C'_U$ belongs to a unique disjoint cycle, thus has in-degree $1$. Therefore, $w(C)=\prod_{u\in U}x_u$. So $w(\mathcal C'_U)=|\mathcal C'|\left(\prod_{u\in U}x_u\right)$, and \eqref{D'U weight} follows.
\end{proof}

\begin{figure}
    \centering
\begin{tikzpicture}[scale=0.6,decoration={
    markings,
    mark=at position 0.55 with {\arrow{>}}}]

\begin{scope}[shift={(0,0)}]



\node at (0,3.5) {$a$};

\node at (0,0.42) {$b$};

\draw[thick, fill] (0.17,3) arc (0:360: 1.75 mm);

\draw[thick, fill] (0.17,1) arc (0:360: 1.75 mm);

\node at (3.05,3.55) {$a'$};

\node at (3.05,0.45) {$b'$};

\draw[thick, fill] (3.17,3) arc (0:360: 1.75 mm);

\draw[thick, fill] (3.17,1) arc (0:360: 1.75 mm);

\node at (6.15,3.55) {$a''$};

\node at (6.15,0.45) {$b''$};

\draw[thick, fill] (6.17,3) arc (0:360: 1.75 mm);

\draw[thick, fill] (6.17,1) arc (0:360: 1.75 mm);

\draw[thick,postaction={decorate}] (0,3) -- (3,3);

\draw[thick,postaction={decorate}] (3,3) -- (6,3);

\draw[thick,postaction={decorate}] (0,1) -- (3,3);

\draw[thick,postaction={decorate}] (3,1) -- (6,3);

\draw[thick] (6,3) .. controls (6.85,3.25) and (6.85,4.00) .. (6,4.25);

\draw[thick] (6,1) .. controls (6.85,0.75) and (6.85,0.00) .. (6,-0.25);

\draw[thick] (0,4.25) .. controls (1.5,4.75) and (4.5,4.75) .. (6,4.25);

\draw[thick] (0,-0.25) .. controls (1.5,-0.75) and (4.5,-0.75) .. (6,-0.25);

\draw[thick,postaction={decorate}] (0,4.25) .. controls (-0.85,4.05) and (-0.85,3.2) .. (0,3);

\draw[thick,postaction={decorate}] (0,-0.25) .. controls (-1.25,0.25) and (-1.25,2) .. (0,3);

\node at (1.90,3.5) {$f_1$};

\node at (1.2,1.1) {$f_1$};

\node at (5,3.5) {$f_2$};

\node at (4,1.1) {$f_2$};

\node at (7.15, 4.05) {$f_3$};

\node at (7.1, 0.15) {$f_3$};

\draw[line width = 0.5mm, orange!90!black, dotted] (3.25,3.80) ellipse (4.85 and 1.55);

\node at (-3.25,3.75) {$C(D)$};

\draw[line width = 0.5mm, green!60!black, dashed] (3.38,0.33) ellipse (4.85 and 1.45);

\node at (-3.25,0.25) {$D_U$};

\end{scope}

\end{tikzpicture}

    \caption{Example of a cyclic $C(D)$ and in-forest $D_U$ decomposition in Lemma~\ref{Cyclic forest weight decomposition}. Also see  Example~\ref{example_decomposition}.}
    \label{fig_00004}
\end{figure}

\begin{example}
\label{example_decomposition}
Let $k=3$ and $X_1 = \{ a,b\}$, $X_2 = \{ a',b'\}$, and 
$X_3 = \{ a'',b''\}$. Let the choice set be $U=\{a,a',a'' \}$.
Let 
\begin{equation*}
\begin{split}
f_1(a) &= a', \quad  f_2(a')= a'', \quad f_3(a'') = a,  \\
f_1(b) &= a', \quad f_2(b') = a'', \quad f_3(b'') = a.
\end{split}
\end{equation*}
The cyclic part is $a\mapsto a'\mapsto a''\mapsto a$, which is an element of $\mathcal{C}_U$.
The action of $f=(f_1,f_2,f_3)$ on $\{b,b',b''\}$ 
induces an in-forest rooted at $U$: $b\mapsto a'$, $b'\mapsto a''$, $b''\mapsto a$, giving a unique directed trail into the root set.
Since there are no additional cycles, the in-forest rooted at $U$ is an element of $\mathcal{F}_U$. The corresponding functional graph is the disjoint edge union $D = C\oplus F$, where the weight factorization is $w(C) = x_a x_{a'}x_{a''}$ and $w(F) = x_{a'}x_{a''}x_{a} = x_a x_{a'}x_{a''}$. 
Since $w(D)=(x_a x_{a'}x_{a''})^2$, we have $w(D)=w(C(D))w(D_U)$, as well as the bijection between $D$ and $(C(D),D_U)$, in Lemma~\ref{Cyclic forest weight decomposition}. See Figure~\ref{fig_00004}.
\end{example}

\subsection{Matrix-tree calculation}\label{The matrix-tree calculation}

The \defn{directed Laplacian matrix} of $\Gamma$ is a square matrix indexed by the vertices of $X$ with entries given by 
\begin{equation}\label{Entries of L}
\ell_{u,v} = 
\begin{cases} 
e_1(X_{i+1}) &\text{if } u = v \in X_i, \\
\quad -x_v & \text{if } u \in X_i \text{ and } v\in X_{i+1}, \\
\qquad  0 & \text{otherwise.}
\end{cases}
\end{equation}

For each vertex subset $U \subseteq X$, let $\mathcal{L}(U)$ denote the principal submatrix obtained by deleting the rows and columns corresponding to vertices in $U$.

We need the following well-known directed matrix-tree theorem~\cite{CK78}:

\begin{theorem}
\label{Matrix tree theorem}
For any vertex subset $U \subseteq X$, the total weight of the spanning in-forests of $\Gamma$ rooted at $U$ is given by
\[w(\mathcal{F}_U) =\det(\mathcal{L}(U)).\]
\end{theorem}

\begin{lemma}\label{Laplacian calculation}
For any $N$-choice set $U$, the total weight of the rooted in-forests $\mathcal F_U$ is given by
\begin{equation}
\label{FU calculation}
w(\mathcal F_U) = 
\prod_{i=1}^k e_1(X_i)^{n_{i-1}-(N+1)}
\left(
\prod_{j=1}^k e_1(X_j) - \prod_{j=1}^k e_1(X_j\setminus U_j)
\right).
\end{equation}
\end{lemma}

\begin{proof}
By Theorem~\ref{Matrix tree theorem}, we have $w(\mathcal{F}_U) = \det(\mathcal L(U))$, so define $H:=\mathcal L(U)$. To compute $\det (H)$, express $H$ as a block matrix, decompose it, and then apply the matrix determinant lemma.

Order each reduced component $X_i\setminus U_i$ and let $\widehat{\mathbf{x}}_i := (x_v)_{v \in X_i\setminus U_i}$ denote the ordered column vector of variables in $X_i\setminus U_i$. Let $p_i:=|X_i\setminus U_i|$, and let $\mathbf{1}_{p_i}$ denote the column vector of length $p_i$ with one in each coordinate. Then order $X\setminus U$ by sequentially concatenating $X_1\setminus U_1,\ldots,X_k\setminus U_k$.

Divide $H$ into block submatrices $H_{i,j}$ of dimension $p_i\times p_j$. Observe each block $H_{i,j}$ is indexed by $X_i\setminus U_i\times X_j\setminus U_j$, with entries
$(H_{i,j})_{u,v}=\ell_{u,v}$, where $u\in X_i\setminus U_i,$ and $v\in X_j\setminus U_j$.
There are three families of blocks. The entries of the diagonal blocks $H_{i,i}$ are given by
\[
(H_{i,i})_{u,v}=
\begin{cases}
    e_1(X_{i+1}) & \text{if }u=v,\\
    \hspace{7mm} 0 &  \text{if }u\neq v,
\end{cases}\]
so $H_{i,i}=e_1 (X_{i+1}) \I_{p_i}$, where $\I_{p_i}$ is $p_i \times p_i$ identity matrix. The entries of the cyclic super-diagonal blocks $H_{i,i+1}$ are given by
$(H_{i,i+1})_{u,v} = -x_v$,
so $H_{i,i+1} = -\mathbf{1}_{p_i} \widehat{\mathbf{x}}_{i+1}^T$. The remaining blocks $H_{i,j}$ are $0$. This yields the explicit block form
\begin{equation}
H = 
\begin{pmatrix}
	e_1(X_2) \I_{p_1} & -\mathbf{1}_{p_1} \widehat{\mathbf{x}}_2^T & 0 & \ldots & 0 & 0 \\
	0 & e_1(X_3) \I_{p_2} & -\mathbf{1}_{p_2} \widehat{\mathbf{x}}_3^T & \ldots & 0 & 0 \\
	\vdots & \vdots & \ddots & \ddots & \vdots & \vdots \\
    0 & 0 & 0 &  e_1 (X_{k-1}) \I_{p_{k-2}}
    & -\mathbf{1}_{p_{k-2}} \widehat{\mathbf{x}}_{k-1}^T  &  0\\ 
	0 & 0 & 0 & \ldots & e_1(X_k) \I_{p_{k-1}} & -\mathbf{1}_{p_{k-1}} \widehat{\mathbf{x}}_k^T \\
	-\mathbf{1}_{p_k} \widehat{\mathbf{x}}_1^T & 0 & 0 & \ldots & 0 & e_1(X_1) \I_{p_k}
\end{pmatrix}.
\end{equation}

We isolate the diagonal blocks of $H$ into a matrix
$D :=\diag(e_1(X_2) \I_{p_1}, \ldots, e_1(X_1) \I_{p_k})$,
and observe that $H=D-ACB^T$, where $C$ is the $k\times k$ cyclic permutation matrix
\[C=
\begin{pmatrix}
0 & 1 & 0 & \cdots & 0 \\
0 & 0 & 1 & \cdots & 0 \\
\vdots & \vdots & \ddots & \ddots & \vdots \\
0 & 0 & 0 & \cdots & 1 \\
1 & 0 & 0 & \cdots & 0
\end{pmatrix},
\] 
and 
$A = \diag(\mathbf{1}_{p_1}, \ldots, \mathbf{1}_{p_k})$ and $B =\diag(\widehat{\mathbf{x}}_1, \ldots, \widehat{\mathbf{x}}_k)$ are block-diagonal matrices.
Applying the matrix determinant lemma (also known as Sylvester's determinant identity), see e.g, ~\cite{HJ13,HJ94,GR01}, gives
\begin{equation}\label{Matrix determinant lemma}
\det(H)=\det(D - A C B^T) = \det(D)\det(\I_k - B^T D^{-1} A C),
\end{equation}
where $\I_k$ is the $k\times k$ identity matrix.

Next, we calculate the determinants of matrices $D$ and $M = \I_k-B^TD^{-1}AC$ in \eqref{Matrix determinant lemma}. The determinant of the diagonal block matrix $D$ is
\begin{equation}\label{Det(D)}
    \det(D) =\prod_{i=1}^k \det(e_1(X_{i+1}) \I_{p_i}) = \prod_{i=1}^k e_1(X_{i+1})^{p_i} = \prod_{i=1}^k e_1(X_i)^{p_{i-1}}.
\end{equation}
Observe that each $\mathbf{1}_{p_i}^T \widehat{\mathbf{x}}_i = e_1(X_i\setminus U_i)$, so the diagonal matrix $B^TD^{-1}A$ simplifies as 
\begin{align*}
B^TD^{-1}A &= \diag\left( \widehat{\mathbf{x}}_i^T (e_1(X_{i+1}) \I_{p_i})^{-1} \mathbf{1}_{p_i} \right)_{i=1}^k\\
&= \diag\left( \frac{e_1(X_1\setminus U_1)}{e_1(X_2)}, \frac{e_1(X_2\setminus U_2)}{e_1(X_3)}, \ldots, \frac{e_1(X_k\setminus U_k)}{e_1(X_1)} \right),
\end{align*}
and $M$ has the form
\[
M=
\begin{pmatrix}
1 & -\frac{e_1(X_1\setminus U_1)}{e_1(X_2)} & 0 & \cdots & 0 & 0 \\
0 & 1 & -\frac{e_1(X_2\setminus U_2)}{e_1(X_3)} & \cdots & 0 & 0 \\
\vdots & \vdots & \ddots & \ddots & \vdots & \vdots \\
0 & 0 & 0 & \cdots & 1 & -\frac{e_1(X_{k-1}\setminus U_{k-1})}{e_1(X_k)} \\
-\frac{e_1(X_k\setminus U_k)}{e_1(X_1)} & 0 & 0 & \cdots & 0 & 1
\end{pmatrix}.
\]
We can calculate $\det(M)$ by direct application of the Leibniz formula:
\[
\det(M)=\sum_{\sigma\in S_k}\sgn(\sigma)M_{1,\sigma(1)}\cdots M_{k,\sigma(k)},
\]
where $\sigma$ is a permutation of the elements of the symmetric group $S_k$ and $\sgn(\sigma)$ is the sign of $\sigma$.
Only two permutations correspond to nonzero summands. The identity $\diag(1,\ldots, 1)$ indexes the diagonal product $(1)^k$ and has sign $1$, and the $k$-cycle $j \mapsto j+1$ indexes the cyclic super-diagonal product 
$\prod_{j=1}^k \left(-\frac{e_1(X_j\setminus U_j)}{e_1(X_{j+1})}\right)$ 
and has sign $(-1)^{k-1}$. Thus 
\begin{equation}\label{det(Ik-MC)}
\det(M) = 1^k+(-1)^{k-1}\prod_{j=1}^k
\left(-\frac{e_1(X_j\setminus U_j)}{e_1(X_{j+1})}\right) 
= 1 - \frac{\prod_{j=1}^k e_1(X_j\setminus U_j)}{\prod_{j=1}^k e_1(X_j)}.
\end{equation}
Returning to \eqref{Matrix determinant lemma}, multiplying $\det(D)$ and $\det(M)$ yields
\begin{align*}
    \det (H)&= 
    \prod_{i=1}^k e_1(X_i)^{p_{i-1}} 
    \left( 1 - \frac{\prod_{j=1}^k e_1(X_j\setminus U_j)}{\prod_{j=1}^k e_1(X_j)} \right)\\
    &=
    \prod_{i=1}^k e_1(X_i)^{p_{i-1}-1} 
    \left( 
    \prod_{j=1}^k e_1(X_j) - \prod_{j=1}^k e_1(X_j\setminus U_j)
    \right).
\end{align*}
Substituting $w(\mathcal F_U)=\det (H)$ and $n_{i-1}-(N+1)=|X_{i-1}\setminus U_{i-1}|-1=p_{i-1}-1$ yields \eqref{FU calculation}.
\end{proof}

\section{Multisymmetric polynomials associated to eventually constant functions}
\label{Eventually constant functions}

    A $k$-tuple of cyclic functions $(f_1,\ldots,f_k)$ with $f_i:X_i\to X_{i+1}$ is \defn{eventually constant} if iterating the composition $f_k \cdots f_1: X_1 \to X_1$ eventually stabilizes to a constant map $(f_k\cdots f_1)^m$. We denote the set of all eventually constant $k$-tuples as $\EC$, and define the generating polynomial $P(\mathbf x)$ as
    \[
    P(\mathbf x) :=\sum_{(f_1,\ldots,f_k)\in\EC}w(f_1,\ldots,f_k)=w(\EC).
    \]

    Recall that a \defn{choice set} $U\subseteq X$ is a vertex subset with exactly $1$ element in each component $U\cap X_i$, denoted $u_i$.

\subsection{Choice set decomposition}
\label{Choice set decomposition}

We consider the subsets $\mathcal D_U$, $\mathcal C_U$, and $\mathcal F_U$, indexed over all choice sets $U$.

For each choice set $U$, $\mathcal C_U$ is the set of cyclic subgraphs of $\Gamma$ with vertex set $U$. Since $u_1$ is a single vertex, there is only one digraph in $\mathcal C_U$, the cycle along the vertices
\[u_1,\,u_2,\,\ldots,\,u_{k-1},\,u_k.\]
Thus, $\mathcal D_U$ is the set of functional subgraphs of $\Gamma$ with a unique cycle on $U$. The set $\mathcal F_U$ remains the set of subforests of $\Gamma$ rooted at $U$, which has weight
\[
w(\mathcal F_U) 
=  \prod_{i=1}^k e_1(X_i)^{n_{i-1}-2} 
\left(
\prod_{j=1}^k e_1(X_j) - \prod_{j=1}^k e_1(X_j\setminus \{u_j\})
\right)
\]
by Lemma \ref{Laplacian calculation}.

By Lemma \ref{Cyclic forest weight decomposition}, each weight $w(\mathcal D_U)$ factors by the weights of $\mathcal C_U$ and $\mathcal F_U$, giving
\begin{equation}\label{DU weight}
\begin{aligned}
     w(\mathcal D_U)&=|\mathcal C_U|\left(\prod_{i=1}^kx_{u_i}\right)w(\mathcal F_U)\\
     &=\left(
     \prod_{i=1}^kx_{u_i}\right)
     \left(\prod_{i=1}^k e_1(X_i)^{n_{i-1}-2}\right)
     \left(
     \prod_{j=1}^k e_1(X_j) - \prod_{j=1}^k e_1(X_j\setminus \{u_j\})
     \right).
\end{aligned}
\end{equation}

Next, we re-express $P(\mathbf x)$ as a sum over the weights $w(\mathcal D_U)$.

\begin{lemma}\label{EC cycle lemma}\label{ECU and FU bijection}
Ranging over all choice sets $U$, the classes $\mathcal D_U$ partition $\EC$.
\end{lemma}
\begin{proof}
Let $(f_1,\ldots,f_k)$ be a spanning functional subgraph of $\Gamma$ with induced graph $D$. Then, $D\in\EC$ if and only if $f_k\cdots f_1$ stabilizes to some singleton on $X_1$. By the remark following Lemma \ref{Cycle bijection lemma}, this occurs if and only if $D$ has a unique cycle whose vertex set is some choice set $U$, precisely if $D\in\mathcal D_U$. Therefore, $\EC=\bigcup_U\mathcal D_U$.

The sets $\mathcal D_U$ are pairwise disjoint since the cyclic part of any digraph $D$ has a unique vertex set $V(C(D))$. Thus, the sets induce a partition.
\end{proof}

By \eqref{DU weight} and the partition, we can express $P(\mathbf x)$ as the following sum over choice sets $U$
\begin{equation}\label{P(x) is a sum over choice sets}
\begin{aligned}
    P(\mathbf x)= \sum_U\left(\prod_{i=1}^kx_{u_i}\right)\left(\prod_{i=1}^k e_1(X_i)^{n_{i-1}-2}\right) 
    \left(
    \prod_{j=1}^k e_1(X_j) - \prod_{j=1}^k e_1(X_j\setminus \{u_j\})
    \right).
\end{aligned}
\end{equation}

\subsection{Multisymmetric polynomials of eventually constant functions}\label{The weight of eventually constant functions}

We now prove the following two results which are stated in Theorem~\ref{Theorem statement intro} and \eqref{eq_P_char}.

\begin{theorem}
\label{EC weight theorem}
The weighted enumeration of all eventually constant $k$-tuples is given by
\begin{equation}
\label{eqn_multisymm_eventually_const_k}
P(\mathbf{x}) 
= \prod_{i=1}^k e_1(X_i)^{n_{i-1}-2}
\left(
\prod_{j=1}^k e_1(X_j)^2 - 2^k\prod_{j=1}^k e_2(X_j)
\right).
\end{equation}
\end{theorem}

\begin{proof}
First, we distribute the cycle weights in the choice set-sum \eqref{P(x) is a sum over choice sets}:
\begin{equation}\label{Choice set forest weight decomposition}
    \begin{aligned}
    P(\mathbf{x}) &=\prod_{i=1}^k e_1(X_i)^{n_{i-1}-2} \sum_U \left(
    \prod_{j=1}^k x_{u_j} e_1(X_j) -\prod_{j=1}^k x_{u_j} e_1(X_j\setminus \{u_j\})
    \right).
\end{aligned}
\end{equation}

By distributivity, we can re-express the sums over choice sets as
\[
\sum_U 
\prod_{j=1}^k x_{u_j} e_1(X_j) 
= \prod_{j=1}^k\sum_{u\in X_j}x_ue_1(X_j), \quad
\sum_U \prod_{j=1}^k x_{u_j} e_1(X_j\setminus \{u_j\}) 
= \prod_{j=1}^k\sum_{u\in X_j} x_{u} e_1(X_j \setminus \{u\}).
\]
Next, we reduce the resulting factors to elementary symmetric polynomials 
\[
\sum_{u \in X_i} x_u e_1(X_i) = e_1(X_i)^2,\quad\text{and}\quad \sum_{u \in X_i} x_u e_1(X_i\setminus \{u\}) = \sum_{u \in X_i} \sum_{\substack{v \in X_i \\ v \neq u}} x_u x_v=2e_2(X_i).
\]
Substituting back into \eqref{Choice set forest weight decomposition} yields the result.
\end{proof}

\begin{corollary}
\label{rep_theory_formulation}
Let $G=\prod_{i=1}^k \GL(n_i)$ and for each $i$, let $V_i\cong \mathbb{C}^{n_i}$ denote the fundamental representation
of $\GL(n_i)$. Then $P(\mathbf x)$ can be written
as the virtual character 
\[
P(\mathbf{x})
=
\chi\!\left(
\bigotimes_{i=1}^k V_i^{\otimes n_{i-1}}
-
2^k
\bigotimes_{i=1}^k
\left(
V_i^{\otimes (n_{i-1}-2)}\otimes \Lambda^2V_i
\right)
\right),
\]
with indices again understood mod k.
\end{corollary}

\begin{proof}
From Theorem~\ref{EC weight theorem}, we have 
\[
P(\mathbf{x})
=
\prod_{i=1}^k e_1(X_i)^{\,n_{i-1}-2}
\left(
\prod_{j=1}^k e_1(X_j)^2
-
2^k\prod_{j=1}^k e_2(X_j)
\right).
\]
We use the standard identities
\[
e_1(X_j) = \chi(V_j),\quad
e_1(X_j)^2 = \chi(V_j^{\otimes 2}),\quad 
\mbox{ and } \quad 
e_2(X_j) = \chi(\Lambda^2V_j)
\]
to obtain 
\[
P(\mathbf{x})
=
\chi\!\left(
\bigotimes_{i=1}^k V_i^{\otimes (n_{i-1}-2)}
\otimes
\left(
\bigotimes_{i=1}^k V_i^{\otimes 2}
-
2^k \bigotimes_{i=1}^k \Lambda^2V_i 
\right)
\right).
\]
Combining tensor powers yields
\[
P(\mathbf{x})
=
\chi\!\left(
\bigotimes_{i=1}^k V_i^{\otimes n_{i-1}}
-
2^k
\bigotimes_{i=1}^k
\left(
V_i^{\otimes (n_{i-1}-2)}\otimes \Lambda^2V_i
\right)
\right),
\]
as claimed.
\end{proof}

\begin{remark}
If any $n_j=1$ then all $k$-tuples $(f_1,\dots, f_k)$ are eventually constant. Furthermore, $\prod_{i=1}^k e_2(X_i) = 0$, and $P(\mathbf{x})$ can be expressed with non-negative coefficients 
\[
P(\mathbf{x}) = \prod_{i=1}^k e_1(X_i)^{n_{i-1}}.
\]
Also, $V_j \cong \mathbb{C}$ and hence $\Lambda^2 V_j = 0$, so
\[
P(\mathbf{x})=\chi\!\left(\bigotimes_{i=1}^k V_i^{\otimes n_{i-1}}\right).
\]
\end{remark}

\subsection{Enumeration of eventually constant functions}
\label{The cardinality of eventually constant functions}

In this section, we enumerate the number of eventually constant functions for a $k$-tuple.

\begin{proposition}\label{EC Cardinality}
The cardinality of $\EC$ is
\[ \#\EC = 
\prod_{i=1}^k n_i^{n_{i-1}-1} 
\left(
\prod_{j=1}^k n_j - \prod_{j=1}^k (n_j-1)
\right). 
\]
\end{proposition}
The proof of the above proposition is a straightforward computation using Theorem \ref{EC weight theorem}, evaluating each variable $x_v$ at 1.


\begin{corollary}
The number of $k$-tuples $(f_1,\ldots,f_k)$ that are eventually constant and stabilize to any fixed choice set is
\[ 
\prod_{i=1}^k n_i^{n_{i-1}-2} 
\left(\prod_{j=1}^k n_j - \prod_{j=1}^k (n_j-1)\right).
\]
\end{corollary}
\begin{proof}
By symmetry, the cardinality of the set $\mathcal{D}_U$ is identical for each choice set $U$. Dividing the total cardinality $\#\EC$ in Proposition \ref{EC Cardinality} by the number of choice sets, $\prod_{i=1}^k n_i$, yields the result.
\end{proof}

\begin{corollary}\label{Uniform EC Cardinality}
If all components $X_i$ are of equal cardinality $|X_i|:= n$, then the number of eventually constant $k$-tuples on $X$ becomes
\[ 
n^{k(n-1)} \left(n^k - (n-1)^k\right).
\]
\end{corollary}

\begin{corollary}\label{Uniform EC asymptotic}
For each integer $n \ge 1$, let $p_k(n)$ denote the probability that a  $k$-tuple of functions between $k$ sets $X_1,\ldots,X_k$ of cardinality $|X_i|=n$ is eventually constant. Then
\begin{equation}\label{Exact EC probability}
     p_k(n) =1 - \left(1 - \frac{1}{n}\right)^k.
\end{equation}
It follows that 
\begin{equation}\label{limit EC probability}
    p_k(n) = \frac{k}{n} + \mathcal O\left(\frac{1}{n^2}\right).
\end{equation}
\end{corollary}
\begin{proof}
Fix any integer $n \ge 1$ and sets $X_1,\ldots,X_k$ of cardinality $|X_i| = n$. The total number of $k$-tuples $(f_1,\ldots,f_k)$ of functions with $f_i: X_i \to X_{i+1}$ is $\prod_{i=1}^k n^n = n^{kn}$. Dividing the number of eventually constant $k$-tuples calculated in Corollary~\ref{Uniform EC Cardinality} by this total yields \eqref{Exact EC probability}:
\[ p_k(n) = \frac{n^{k(n-1)} \big(n^k - (n-1)^k\big)}{n^{kn}} = \frac{n^k - (n-1)^k}{n^k} = 1 - \left(1 - \frac{1}{n}\right)^k. \]
Expansion using the binomial theorem gives \eqref{limit EC probability}:
\[ 1-\left(1 - \frac{1}{n}\right)^k = 1-\left(1 - \frac{k}{n} + \mathcal O\left(\frac{1}{n^2}\right)\right)= \frac{k}{n} + \mathcal O\left(\frac{1}{n^2}\right). \]
\end{proof}

For $k=1$, this recovers \eqref{k=1_EC_prob_intro}, the classical probability recently highlighted in \cite{Lei21} as a set-theoretic analogue to nilpotent linear operators.

As expected, if we fix $n \ge 1$ instead, then the chance of picking an eventually constant $k$-tuple grows as $k$ does, approaching $1$ as $k\to\infty$.

Finally, if we vary both parameters $k$ and $n$ together, we recover the exponential function. For any fixed positive constant $c\in\mathbb N$, we see
\[
\lim_{n\to\infty}p_{cn}(n)=1-\lim_{n\to\infty}\left(\left(1-\frac{1}{n}\right)^{cn}\right)=1-e^{-c}.
\]

\subsection{Eventually constant pairs and triples}
\label{Specializations: k=2 and k=3}

In this section, we discuss special cases when $k=2,3$.

\begin{example}
    Suppose $k=2$ and that $X_1,\,X_2$ are disjoint sets of cardinality $n_1$ and $n_2$. Then, an eventually constant pair $(f, g)$ consists of maps $f: X_1 \to X_2$ and $g: X_2 \to X_1$ such that some iterate $(gf)^m: X_1 \to X_1$ stabilizes to a constant map on $X_1$, terminating at some point $v_1 \in X_1$. In this case, the digraph $\Gamma$ reduces to the directed bipartite graph between $X_1$ and $X_2$.

We can visualize the pair $(f,g)$ by its induced bipartite digraph $D$. As shown in Lemma \ref{EC cycle lemma}, the collapse to a constant function ensures that $D$ is weakly connected and it contains a unique $2$-cycle between $v_1$ and $f(v_1)$. Removing the edges of this cycle yields a spanning in-forest of $\Gamma$ rooted at $\{v_1, f(v_1)\}$. See Figure~\ref{fig_00001}.

\begin{figure}
    \centering
\begin{tikzpicture}[scale=0.6,decoration={
    markings,
    mark=at position 0.35 with {\arrow{>}}}]

\begin{scope}[shift={(0,-2)}]


    \node at (2,2.75) {$g$};

    \draw[thick, postaction = {decorate}] (4,2) .. controls (3.5,2.5) and (0.5,2.5) .. (0,2);

    \draw[thick, postaction = {decorate}] (0,2) .. controls (0.5,1.25) and (3.5,1.25) .. (4,2);

    \node at (2,1.00) {$f$};

    \draw[thick, fill] (0.25,2) arc (0: 360: 1.75 mm);

    \node at (-0.65,2) {$x_1$};

    \draw[thick, fill] (4.25,2) arc (0: 360: 1.75 mm);

    \node at (4.75,2) {$y_1$};

\end{scope}


\begin{scope}[shift={(12,0)}]

\begin{scope}[shift={(0,2)}]

    \node at (3.00,3.45) {$g$};

    \draw[thick, postaction = {decorate}] (4,3) .. controls (3.5,3.35) and (0.5,3.10) .. (0,0);

    \draw[thick, fill] (4.25,3) arc (0: 360: 1.75 mm);

    \node at (4.75,3) {$y_1$};

\end{scope}

\begin{scope}[shift={(0,0)}]


    \draw[thick, postaction = {decorate}] (0,2) .. controls (0.5,1.25) and (3.5,2.25) .. (4,3);

    \node at (2,2.35) {$f$};

    \draw[thick, postaction = {decorate}] (4,3) .. controls (3.5,4.35) and (0.5,4.10) .. (0,4);

    \node at (-0.25,4.35) {$g$};

    \draw[thick, postaction = {decorate}] (0,4) .. controls (-2,3.75) and (-3,-1) .. (0,-2);

    \draw[thick, fill] (0.25,2) arc (0: 360: 1.75 mm);

    \node at (-0.65,2) {$x_1$};

    \draw[thick, fill] (4.25,3) arc (0: 360: 1.75 mm);

    \node at (4.75,3) {$y_2$};
        
\end{scope}

\begin{scope}[shift={(0,-2)}]


    \draw[thick, postaction = {decorate}] (0,2) .. controls (0.5,1.25) and (3.5,2.25) .. (4,3);

    \node at (1.75,1.45) {$f$};

    \draw[thick, fill] (0.25,2) arc (0: 360: 1.75 mm);

    \node at (-0.65,2) {$x_2$};

    \draw[thick, fill] (4.25,3) arc (0: 360: 1.75 mm);

    \node at (4.75,3) {$y_3$};

    \draw[thick, postaction = {decorate}] (4,3) .. controls (3.5,3.5) and (0.5,3) .. (0,2);

    \node at (3.5,3.50) {$g$};

\end{scope}

\begin{scope}[shift={(0,-4)}]


    \draw[thick, postaction = {decorate}] (0,2) .. controls (0.5,1.55) and (3.5,2.25) .. (4,3);

    \node at (1.5,1.50) {$f$};

    \draw[thick, postaction = {decorate}] (4,3) .. controls (3.5,4.1) and (0.5,5) .. (0,4);

    \node at (3.65,3.90) {$g$};

    \draw[thick, fill] (0.25,2) arc (0: 360: 1.75 mm);

    \node at (-0.65,1.75) {$x_3$};

    \draw[thick, fill] (4.25,3) arc (0: 360: 1.75 mm);

    \node at (4.75,3) {$y_4$};

\end{scope}

\begin{scope}[shift={(0,-6)}]


    \node at (1.5,1.65) {$f$};

    \draw[thick,postaction={decorate}] (0,2) .. controls (0.5,2) and (3,2) .. (4,3);

    \draw[thick, postaction = {decorate}] (4,3) .. controls (3.5,3.25) and (0.5,5) .. (0,6);

    \node at (3.5,3.70) {$g$};

    \draw[thick, fill] (0.25,2) arc (0: 360: 1.75 mm);

    \node at (-0.65,1.75) {$x_4$};

    \draw[thick, fill] (4.25,3) arc (0: 360: 1.75 mm);

    \node at (4.75,3) {$y_5$};

\end{scope}

\end{scope}

\end{tikzpicture}

    \caption{Left: example of an eventually constant pair.
    Right: example of an eventually constant pair with $|X_1|=4$ and $|X_2|=5$.}
    \label{fig_00001}
\end{figure}

By Theorem~\ref{EC weight theorem} and Corollary~\ref{rep_theory_formulation}, the generating polynomial and the virtual character for the pairs, respectively, are
\begin{align*}
    P(\mathbf{x})&= e_1(X_1)^{n_2-2} e_1(X_2)^{n_1-2} \left( e_1(X_1)^2 e_1(X_2)^2 - 4 e_2(X_1)e_2(X_2) \right)\\
&=\chi\!\left(
\left(V_1^{\otimes (n_2-2)} \otimes V_2^{\otimes (n_1-2)}\right)
\otimes
\left(
V_1^{\otimes 2}\otimes V_2^{\otimes 2}
\;-\;
4\,\Lambda^2 V_1 \otimes \Lambda^2 V_2
\right)
\right).
\end{align*}

Applying Proposition~\ref{EC Cardinality}, we recover the cardinality of eventually constant pairs discovered in~\cite{CIKLR25,CILR25}:
\[
|(n_1, n_2)| = n_1^{n_2-1} n_2^{n_1-1}(n_1 +n_2-1), 
\]
where $\#\EC = |(n_1, n_2)|$.
\end{example}

\begin{example}
Let $k=3$, with $X_1, X_2, X_3$ finite disjoint sets of cardinalities $n_1, n_2, n_3$, respectively, and consider an eventually constant triple $(f,g,h)$ that collapses to a point $v_1\in X_1$. The digraph $\Gamma$ has directed edges from $X_1$ to $X_2$, $X_2$ to $X_3$, and $X_3$ to $X_1$, and $(f,g,h)$ induces a subgraph $D$ containing a unique $3$-cycle around $v_1,\,f(v_1),\,gf(v_1)$. Removing the edges of this cycle again yields an in-forest of $\Gamma$ rooted around this cycle. 
See Figure~\ref{fig_00002}.

\begin{figure}
    \centering
\begin{tikzpicture}[scale=0.6,decoration={
    markings,
    mark=at position 0.35 with {\arrow{>}}}]

\begin{scope}[shift={(0,-2)}]


    \node at (4,3.55) {$h$};

    \draw[thick, postaction = {decorate}] (8,2) .. controls (7,3.5) and (1,3.5) .. (0,2);

    \draw[thick, postaction = {decorate}] (0,2) .. controls (0.5,1.25) and (3.5,1.25) .. (4,1.5);

    \node at (2,0.90) {$f$};

    \draw[thick, fill] (0.25,2) arc (0:360: 1.75 mm);

    \node at (-0.65,2) {$x_1$};

    \draw[thick, fill] (4.25,1.5) arc (0:360: 1.75 mm);

    \node at (4.05,0.95) {$y_1$};

\begin{scope}[shift={(4,0)}]

    \draw[thick, postaction = {decorate}] (0,1.5) .. controls (0.5,1.25) and (3.5,1.25) .. (4,2);

    \node at (2.25,1.00) {$g$};

\end{scope}

    \draw[thick, fill] (8.25,2) arc (0: 360: 1.75 mm);

    \node at (8.75,2) {$z_1$};

\end{scope}


\begin{scope}[shift={(13,0)}]

\begin{scope}[shift={(0,2)}]

    \draw[thick, fill] (4.25,3) arc (0: 360: 1.75 mm);

    \node at (4.0,2.45) {$y_1$};

\end{scope}

\begin{scope}[shift={(0,0)}]


    \draw[thick, postaction = {decorate}] (0,2) .. controls (0.5,1.25) and (3.5,2.25) .. (4,3);

    \node at (2,2.35) {$f$};

    \draw[thick, postaction = {decorate}] (0,4.5) .. controls (-1.5,3.3) and (-2.5,-1) .. (0,-2);

    \draw[thick, fill] (0.25,2) arc (0: 360: 1.75 mm);

    \node at (-0.55,2) {$x_1$};

    \draw[thick, fill] (4.25,3) arc (0: 360: 1.75 mm);

    \node at (4.05,2.35) {$y_2$};
        
\end{scope}

\begin{scope}[shift={(0,-2)}]


    \draw[thick, postaction = {decorate}] (0,2) .. controls (0.5,1.25) and (3.5,2.25) .. (4,3);

    \node at (2,2.40) {$f$};

    \draw[thick, fill] (0.25,2) arc (0: 360: 1.75 mm);

    \node at (-0.55,2) {$x_2$};

    \draw[thick, fill] (4.25,3) arc (0: 360: 1.75 mm);

    \node at (4.05,2.35) {$y_3$};

\end{scope}

\begin{scope}[shift={(0,-4)}]


    \draw[thick, postaction = {decorate}] (0,2) .. controls (0.5,1.55) and (3.5,2.25) .. (4,3);

    \node at (2,2.50) {$f$};

    \draw[thick, fill] (0.25,2) arc (0: 360: 1.75 mm);

    \node at (-0.5,1.75) {$x_3$};

    \draw[thick, fill] (4.25,3) arc (0: 360: 1.75 mm);

    \node at (4.1,2.45) {$y_4$};

\end{scope}

\begin{scope}[shift={(0,-6)}]


    \node at (2,2.60) {$f$};

    \draw[thick,postaction={decorate}] (0,2) .. controls (0.5,2) and (3,2) .. (4,3);

    \draw[thick, fill] (0.25,2) arc (0: 360: 1.75 mm);

    \node at (-0.5,1.75) {$x_4$};

    \draw[thick, fill] (4.25,3) arc (0: 360: 1.75 mm);

    \node at (4.1,2.45) {$y_5$};

\end{scope}

\begin{scope}[shift={(4,0)}]


    \node at (2.5,3.80) {$g$};

    \draw[thick, postaction = {decorate}] (0,5) .. controls (0.5,4) and (3.5,3) .. (4,3);

    \node at (2,2) {$g$};

    \draw[thick,postaction={decorate}] (0,3) .. controls (0.5,2) and (3,1) .. (4,1);

    \draw[thick,postaction={decorate}] (4,3) .. controls (3,7) and (-2,6) .. (-4,4.5);

    \node at (3.4,5) {$h$};

    \draw[thick, fill] (4.25,3) arc (0: 360: 1.75 mm);

    \node at (4.65,2.8) {$z_1$};

\end{scope}

\begin{scope}[shift={(4,-2)}]


    \node at (2.1,1.9) {$g$};

    \draw[thick,postaction={decorate}] (0,3) .. controls (0.5,2) and (3,1) .. (4,1);

    \node at (1.1,-2.00) {$h$};

    \draw[thick,postaction={decorate}] (4,3) .. controls (6.25,2) and (6.25,-1) .. (4,-2.5);

    \node at (6.00,0.0) {$h$};

    \draw[thick,postaction={decorate}] (4,-2.5) .. controls (2,-4) and (-2,-4) .. (-4,-2);

    \draw[thick,postaction={decorate}] (4,1) .. controls (3,-4) and (-3,-3) .. (-4,-2);

    \draw[thick, fill] (4.25,3) arc (0: 360: 1.75 mm);

    \node at (4.65,3.25) {$z_2$};

\end{scope}

\begin{scope}[shift={(4,-4)}]


    \node at (1.8,2.95) {$g$};

    \draw[thick,postaction={decorate}] (0,3) .. controls (0.5,2.5) and (3,2.5) .. (4,3);

    \node at (2,1.85) {$g$};

    \draw[thick,postaction={decorate}] (0,1) .. controls (0.5,1) and (3,1.5) .. (4,3);

    \draw[thick, fill] (4.25,3) arc (0: 360: 1.75 mm);

    \node at (4.60,3.25) {$z_3$};

\end{scope}

\end{scope}

\end{tikzpicture}

    \caption{Left: example of an eventually constant triple.
    Right: example of an eventually constant triple with $|X_1|=4$, $|X_2|=5$, and $|X_3|=3$.}
    \label{fig_00002}
\end{figure}

By Theorem \ref{EC weight theorem}, the generating polynomial is 
\begin{align*}
P(\mathbf{x})&=e_1(X_1)^{n_3-2} e_1(X_2)^{n_1-2} e_1(X_3)^{n_2-2} \left( \prod_{i=1}^3 e_1(X_i)^2 - 8 \prod_{i=1}^3 e_2(X_i) \right),
\end{align*}
which can be expressed as 
\[\chi\!\left(
\left(
V_1^{\otimes (n_3-2)}
\otimes
V_2^{\otimes (n_1-2)}
\otimes
V_3^{\otimes (n_2-2)}
\right)
\otimes
\left(
V_1^{\otimes 2}\otimes V_2^{\otimes 2}\otimes V_3^{\otimes 2}
-
8\,\Lambda^2 V_1 \otimes \Lambda^2 V_2 \otimes \Lambda^2 V_3
\right)
\right)
\]
by Corollary~\ref{rep_theory_formulation}.
Using Proposition~\ref{EC Cardinality}, we obtain the enumeration for eventually constant triples:
\[ 
\#\EC = 
|(n_1, n_2, n_3)| = n_1^{n_3-1} n_2^{n_1-1} n_3^{n_2-1} (n_1 n_2 + n_2 n_3 + n_3 n_1 - (n_1 + n_2 + n_3) + 1).
\]
\end{example}

\section{Multisymmetric polynomials associated to eventually cyclic functions}
\label{Eventually cyclic functions}

Fix a positive integer $N\in\mathbb N$. A $k$-tuple $(f_1,\ldots,f_k)$ is \defn{eventually $N$-cyclic} if iterated images of composition $f_k\cdots f_1:X_1\to X_1$ eventually stabilize to a subset $(f_k\cdots f_1)^m(X_1)\subseteq X_1$ of cardinality $N$, and $f_k\cdots f_1$ acts as an $N$-cycle on this set. Let $\EC(N)$ be the set of eventually $N$-cyclic tuples, and let $P_N(\mathbf x)$ be the generating polynomial
    \[
    P_N(\mathbf x):=\sum_{D\in\EC(N)}w(D)=w(\EC(N)).
    \]
    Note that $\EC=\EC(1)$. 

Recall, an \defn{$N$-choice set} $U\subseteq X$ is a vertex subset that contains exactly $N$ elements in each component, denoted $U_i:=X_i\cap U$.

\subsection{\texorpdfstring{$N$}{N}-choice set decomposition}
\label{N-choice set decomposition}

To calculate $P_N(\mathbf x)$, we adapt the methods in Section~\ref{Choice set decomposition} for eventually constant functions. Generalizing, we shift from standard choice sets to $N$-choice sets $U$, and work with subsets $\mathcal C_U^{(N)}$ and $\mathcal D_U^{(N)}$ of $\mathcal C_U$ and $\mathcal D_U$. Then as before, we weigh each subset $\mathcal D_U^{(N)}$ by Lemma \ref{Cyclic forest weight decomposition} and partition the set $\EC(N)$ by $\mathcal D_U^{(N)}$.

We restrict the sets $\mathcal D_U$ and $\mathcal C_U$ to full cycles. Discarding disjoint unions of distinct cycles in $\mathcal C_U$, we consider only the subset $\mathcal C_U^{(N)}$ of (single) cycle subgraphs of $\Gamma$ on $U$. Accordingly, we define the subset $\mathcal D_U^{(N)}\subseteq\mathcal D_U$ of spanning $\Gamma$ subgraphs whose cyclic part belongs to $\mathcal C_U^{(N)}$, or equivalently, whose cyclic part is a (single) cycle on $U$.

Note that although $\mathcal D_U^{(N)}$ may be a proper subset of $\mathcal D_U$, Lemma \ref{Cycle bijection lemma} still allows us to calculate each weight $w(\mathcal D^{(N)}_U)$ in terms of the corresponding subset $\mathcal C_U^{(N)}\subseteq\mathcal C_U$ and $\mathcal F_U$:
\begin{equation}\label{DUN weight}
    w(\mathcal D_U^{(N)})=|\mathcal C_U^{(N)}|\left(\prod_{u\in U}x_u\right)w(\mathcal F_U).
\end{equation}

\begin{lemma}\label{Generalized cycle lemma}
Ranging over all $N$-choice sets $U$, the sets $\mathcal{D}_U^{(N)}$ partition $\EC(N)$.
\end{lemma}
\begin{proof}
Let $D$ be a spanning functional subgraph of $\Gamma$ with corresponding $k$-tuple $(f_1,\ldots,f_k)$. Then, $D \in \EC(N)$ if and only if the cyclic part of its composition $f_k \cdots f_1$ is a single $N$-cycle on some subset $U_1\subseteq X_1$ of cardinality $N$. By remark following Lemma \ref{Cycle bijection lemma}, this occurs if and only if $D$ contains a unique cycle whose vertex set forms an $N$-choice set extending $U_1$, which holds if and only if $D\in\bigcup_U\mathcal{D}_U^{(N)}$.

This induces a partition since the sets $\mathcal{D}_U^{(N)}$ are pairwise disjoint.
\end{proof}

\begin{lemma}
\label{Cycle enumeration lemma}
The cardinality of $\mathcal C^{(N)}_U$ is
\[
|\mathcal{C}^{(N)}_U| = \frac{(N!)^k}{N}.
\]
\end{lemma}

\begin{proof}
By Proposition~\ref{Cycle bijection lemma}, $\mathcal{C}^{(N)}_U$ identifies with the set of $k$-tuples $(f_1, \ldots,f_k)$, such that each map $f_i: U_i \to U_{i+1}$ is a bijection and $f_k\cdots f_1:U_1\to U_1$ is an $N$-cycle. 

By right-composition on the final term,
\[
(f_1,\ldots,f_k)\mapsto(f_1,\ldots,f_{k-1},f_k\cdots f_1),
\]
bijects $\mathcal C_U^{(N)}$ with the $k$-tuples $(f_1,\ldots,f_{k-1},g)$ for which each $f_i:U_i\to U_{i+1}$ is a bijection and $g:U_1\to U_1$ is an $N$-cycle. There are $N!$ such bijections $f_i:U_i\to U_{i+1}$ for each component $i=1,\ldots,k-1$, and there are $(N-1)!$ such $N$-cycles on $U_1$.
\end{proof}

By the partition over $\mathcal D_U^{(N)}$ in Lemma \ref{Generalized cycle lemma}, the cardinality in Lemma~\ref{Cycle enumeration lemma}, and \eqref{DUN weight}, 
we obtain
\begin{equation}
\label{Sum over m-choice sets}
\begin{aligned}
        P_N(\mathbf x)&=\sum_U w(\mathcal D_U^{(N)})=\frac{(N!)^k}{N}\sum_U \left(\prod_{u\in U}x_u\right)w(\mathcal F_U).
\end{aligned}
\end{equation}

\subsection{Multisymmetric polynomials of eventually cyclic functions}
\label{subsection_wt_eventually_cyclic_fns}

\begin{theorem}\label{Weight of eventually cyclic functions}
The weighted enumeration of all eventually $N$-cyclic $k$-tuples is given by
\[
P_N(\mathbf{x}) 
= \frac{(N!)^k}{N} 
\prod_{i=1}^k e_1(X_i)^{n_{i-1}-(N+1)}
\left(
\prod_{j=1}^k e_1(X_j)e_N(X_j) - (N+1)^k\prod_{j=1}^k e_{N+1}(X_j)
\right).
\]
\end{theorem}

\begin{proof}
To simplify, let $A_N$ abbreviate the common factored weight
\[A_N :=\frac{(N!)^k}{N} \prod_{i=1}^k e_1(X_i)^{n_{i-1}-(N+1)}.\]

By Lemma \ref{Laplacian calculation}, the weight of the subforests rooted at each $N$-choice set $U$ is
\[
w(\mathcal F_U) 
= \prod_{i=1}^k e_1(X_i)^{n_{i-1}-(N+1)} 
\left(
\prod_{j=1}^k e_1(X_j) - \prod_{j=1}^k e_1(X_j\setminus U_j)
\right).
\]
We apply the forest weights to \eqref{Sum over m-choice sets} and distribute the terms $\prod_{u\in U_i} x_u$ to get the following sum over $N$-choice sets $U$:
\begin{equation}\label{Unfactored PN equation}
\begin{aligned}
    P_N(\mathbf{x}) 
    &= A_N\sum_U 
    \prod_{i=1}^k \prod_{u \in U_i} x_u 
    \left(
    \prod_{j=1}^k e_1(X_j) - \prod_{j=1}^k e_1(X_j\setminus U_j)
    \right)\\
    &= A_N \sum_U \left( \prod_{i=1}^k e_1(X_i) \prod_{u \in U_i} x_u - \prod_{i=1}^k e_1(X_i\setminus U_i) \prod_{u \in U_i} x_u \right).
\end{aligned}
\end{equation}
By distributivity, we can re-express the sums over $N$-choice sets as products over $i=1,\ldots,k$:
\begin{align*}\sum_U\prod_{i=1}^k e_1(X_i)\prod_{u \in U_i}x_u&=\prod_{i=1}^k\sum_{\substack{U_i\subseteq X_i\\|U_i|=N}} e_1(X_i)\prod_{u\in U_i}x_u,\\
\sum_U\prod_{i=1}^k e_1(X_i\setminus U_i) \prod_{u \in U_i} x_u &=\prod_{i=1}^k\sum_{\substack{U_i\subseteq X_i\\|U_i|=N}}e_1(X_i\setminus U_i)\prod_{u\in U_i}x_u.
\end{align*}
Next, we reduce the factors of the resulting products to elementary symmetric polynomials. The first factors simplify immediately:
\[\sum_{\substack{U_i\subseteq X_i\\|U_i|=N}} e_1(X_i) \prod_{u \in U_i} x_u = e_1(X_i)e_N(X_i).\]
For the next set of factors, note that there are $N+1$ ways to form each subset $U'_i\subseteq X_i$ of cardinality $N+1$ from a subset $U_i\subseteq X_i$ of cardinality $N$ and an element $v\in X_i\setminus U_i$. Thus, expanding $e_1(X_i\setminus U_i)$ yields
\[\sum_{\substack{U_i\subseteq X_i\\|U_i|=N}} e_1(X_i\setminus U_i) \prod_{u \in U_i} x_u = \sum_{\substack{U_i\subseteq X_i\\|U_i|=N}} \sum_{v \in X_i \setminus U_i} x_v \prod_{u \in U_i} x_u=(N+1) e_{N+1}(X_i).\]
Substituting back into \eqref{Unfactored PN equation} yields
\[P_N(\mathbf{x}) = A_N\left(\prod_{i=1}^k e_1(X_i)e_N(X_i)- (N+1)^k\prod_{i=1}^ke_{N+1}(X_i)\right).
\]
This completes the proof.
\end{proof}

\begin{remark}
If any $n_j <N+1$, then $\prod_{i=1}^k e_{n+1}(X_i) = 0$, and the weight becomes
\[P_N(\mathbf{x}) = (N!)^{k-1}(N-1)!\prod_{i=1}^k e_1(X_i)^{n_{i-1}-N}e_N(X_i).
\]
\end{remark}

\begin{corollary}
\label{N_cyclic_rep_form}
Let $X_1,\dots,X_k$ be finite sets with $|X_i|=n_i \ge 2$, and let  $V_i \cong \mathbb{C}^{n_i}$ be the fundamental representation of $\GL(n_i)$. Then
\begin{equation}
\label{eqn_funda_rep_k_sets}
\begin{split}
P_n(\mathbf{x})
&=
(n!)^{k-1}(n-1)!
\;
\chi\!\left(
\bigotimes_{i=1}^k
\left(
V_i^{\otimes (n_{i-1}-n)} \otimes \Lambda^n V_i
\right)  \right.\\
&\quad -
\left. (n+1)^k
\bigotimes_{i=1}^k
\left(
V_i^{\otimes (n_{i-1}-(n+1))} \otimes \Lambda^{n+1} V_i
\right)
\right),
\end{split}
\end{equation}
where indices are taken modulo $k$.
\end{corollary}

\begin{proof}
We begin from 
\[
P_n(\mathbf{x})
=
(n!)^{k-1}(n-1)!
\left(
\prod_{i=1}^k e_1(X_i)^{\,n_{i-1}-(n+1)}
\right)
\left(
\prod_{i=1}^k e_n(X_i)e_1(X_i)
-
(n+1)^k \prod_{i=1}^k e_{n+1}(X_i)
\right).
\]
Substituting standard character identities
\[
e_1(X_i) = \chi(V_i), \qquad
e_n(X_i) = \chi(\Lambda^n V_i), \qquad
e_{n+1}(X_i) = \chi(\Lambda^{n+1} V_i),
\]
into~\eqref{eqn_funda_rep_k_sets} gives
\[
P_n(\mathbf{x})
=
(n!)^{k-1}(n-1)!
\left(
\prod_{i=1}^k \chi(V_i)^{\,n_{i-1}-(n+1)}
\right)
\left(
\prod_{i=1}^k \chi(\Lambda^n V_i)\chi(V_i)
-
(n+1)^k \prod_{i=1}^k \chi(\Lambda^{n+1} V_i)
\right).
\]

We rewrite each product as a single character:
\begin{equation*}
\begin{split}
\prod_{i=1}^k \chi(V_i)^{\,n_{i-1}-(n+1)}
=
\chi\!\left(
\bigotimes_{i=1}^k V_i^{\otimes (n_{i-1}-(n+1))}
\right), \quad
&\prod_{i=1}^k \chi(\Lambda^n V_i)\chi(V_i) =\chi\!\left( \bigotimes_{i=1}^k (\Lambda^n V_i \otimes V_i)
\right), \\
\prod_{i=1}^k \chi(\Lambda^{n+1} V_i) &= \chi\!\left(
\bigotimes_{i=1}^k \Lambda^{n+1} V_i
\right).
\end{split}
\end{equation*}
Combining these, we obtain
\[
P_n(\mathbf{x})
=
(n!)^{k-1}(n-1)!
\;
\chi\!\left(
\bigotimes_{i=1}^k
V_i^{\otimes (n_{i-1}-(n+1))}
\otimes
\left(
\bigotimes_{i=1}^k (\Lambda^n V_i \otimes V_i)
-
(n+1)^k \bigotimes_{i=1}^k \Lambda^{n+1} V_i
\right)
\right).
\]

Finally, regrouping tensor factors inside each $i$ gives $V_i^{\otimes (n_{i-1}-(n+1))} \otimes V_i = V_i^{\otimes (n_{i-1}-n)}$, so that $V_i^{\otimes (n_{i-1}-(n+1))} \otimes (\Lambda^n V_i \otimes V_i)
=
V_i^{\otimes (n_{i-1}-n)} \otimes \Lambda^n V_i.$ Substituting this yields the claimed expression~\eqref{eqn_funda_rep_k_sets}.
\end{proof}

\subsection{Enumeration of eventually cyclic functions}
\label{Enumeration_eventually_cyclic}

\begin{proposition}\label{prop_card_EC(N)}
The cardinality of $\EC(N)$ is given by 
\begin{equation}\label{ECN geometric form}
    \#\EC(N) = 
    \prod_{i=1}^k \binom{n_i}{N} 
    \frac{(N!)^k}{N}  
     n_i^{n_{i-1}-(N+1)}
    \left(
    \prod_{j = 1}^k n_j - \prod_{j = 1}^k (n_j - N)
    \right),
\end{equation}
or, equivalenty, 
\begin{equation}\label{ECN algebraic form}
    \#\EC(N) = \frac{1}{N}
    \prod_{i=1}^k n_i^{n_{i-1}-N}
     \frac{(n_i-1)!}{(n_i-N)!}
    \left(
    \prod_{j=1}^k n_j - \prod_{j = 1}^k (n_j - N)
    \right).
\end{equation}
\end{proposition}
The proof of the above proposition is a straightforward computation using Theorem \ref{Weight of eventually cyclic functions}, evaluating each variable $x_v$ at 1.




In \eqref{ECN geometric form}, the factor $\prod_{i=1}^k \binom{n_i}{N}$ enumerates the $N$-choice sets; we choose $N$ elements from each component $X_i$. The coefficient $\frac{(N!)^k}{N}$ enumerates the cycle subgraphs on $U$ by \ref{Cycle enumeration lemma}. The third term enumerates the spanning in-forests of $\Gamma$ rooted at $U$, obtained by evaluating the forest weight polynomial $w(\mathcal{F}_U)$ at $x_u=1$.

Also, the cardinality of the subfamily $\mathcal{D}_U^{(N)}$ is identical for each $N$-choice set $U$, and we know there are $\prod_{i=1}^k \binom{n_i}{N}$ many $N$-choice sets. Therefore, the number of $k$-tuples that are eventually $N$-cyclic and stabilize to any single $N$-choice set is 
\[ 
\frac{(N!)^k}{N}  
\prod_{i=1}^k n_i^{n_{i-1}-(N+1)}
\left(\prod_{j=1}^k n_j - \prod_{j=1}^k (n_j-N)\right). 
\]

\begin{corollary}\label{Uniform ECN Cardinality}
If all components $X_i$ are of equal cardinality $|X_i| := n$, then the number of eventually $N$-cyclic $k$-tuples on $X$ becomes 
\[ 
\frac{1}{N} n^{k(n-N)} \left(\frac{(n-1)!}{(n-N)!}\right)^k \left(n^k - (n-N)^k\right). 
\]
\end{corollary}

\begin{corollary}\label{Uniform ECN asymptotic}
For each integer $n \ge N+1$, let $p_{k,N}(n)$ denote the probability that a $k$-tuple of functions between $k$ sets of cardinality $n$ is eventually $N$-cyclic. Then, 
\begin{equation}\label{Exact ECN probability}
    p_{k,N}(n) = \frac{1}{N} n^{-k(N-1)} \left(\frac{(n-1)!}{(n-N)!}\right)^k \left(1 - \left(1 - \frac{N}{n}\right)^k\right),
\end{equation}
and we derive the same limit behavior from Corollary \ref{Uniform EC asymptotic}:
\begin{equation}\label{limit ECN probability}
    p_{k,N}(n) = \frac{k}{n} + \mathcal O\left(\frac{1}{n^2}\right).
\end{equation}
\end{corollary}
\begin{proof}
Fix any integer $n \ge N+1$. We divide the number of eventually $N$-cyclic $k$-tuples calculated in Corollary \ref{Uniform ECN Cardinality} by $n^{kn}$, the number of all $k$-tuples of functions, to get \eqref{Exact ECN probability}:
\begin{align*}
p_{k,N}(n) &= \frac{1}{N} \frac{n^{k(n-N)}}{n^{kn}} \left(\frac{(n-1)!}{(n-N)!}\right)^k \left(n^k - (n-N)^k\right) \\
&= \frac{1}{N} \left( n^{-(N-1)} \frac{(n-1)!}{(n-N)!} \right)^k \left(1 - \left(1 - \frac{N}{n}\right)^k\right).
\end{align*}
For the bracketed term, we distribute a factor $n$ of $n^{N-1}$ into the denominator of each of the $N-1$ factors of $\frac{(n-1)!}{(n-N)!}$ to get the $1+\mathcal O\left(\frac{1}{n}\right)$ term 
\[ 
n^{-(N-1)} \frac{(n-1)!}{(n-N)!}  = \left(\frac{n-1}{n}\right)\left(\frac{n-2}{n}\right)\cdots\left(\frac{n-N+1}{n}\right)= \prod_{j=1}^{N-1} \left(1 - \frac{j}{n}\right). 
\]

Next, expansion of the rightmost parenthesis using the binomial theorem gives 
\[ 
1 - \left(1 - \frac{N}{n}\right)^k = \frac{kN}{n} + \mathcal O\left(\frac{1}{n^2}\right). 
\]
Multiplying these asymptotic expansions together yields \eqref{limit ECN probability}:
\[ 
p_{k,N}(n) = \frac{1}{N} \left( 1 + \mathcal O\left(\frac{1}{n}\right)\right)^k \left(\frac{kN}{n} +\mathcal O\left(\frac{1}{n^2}\right)\right) = \frac{k}{n} + \mathcal O\left(\frac{1}{n^2}\right). 
\]
\end{proof}

\begin{corollary}\label{Proportional ECN asymptotic}
Fix any positive integer $c\in\mathbb N$. As $n \to \infty$, the probability $p_{cn,N}(n)$ that a $cn$-tuple is eventually $N$-cyclic equals
\[ 
\lim_{n \to \infty} p_{cn,N}(n) = \frac{1}{N} e^{-c \binom{N}{2}} \left( 1 - e^{-cN} \right). 
\]
\end{corollary}
\begin{proof}
We substitute $k = cn$ into the exact probability formula \eqref{Exact ECN probability} and analyze the limit of the two factors. First, using the standard limit $\lim_{n\to\infty} \left(1 - \frac{x}{n}\right)^n = e^{-x}$, the rightmost term approaches
\[ 
\lim_{n \to\infty} 
 1 - \left(1 - \frac{N}{n}\right)^{cn} 
= 1 - (e^{-N})^c = 1 - e^{-cN}. 
\]
Secondly, we apply the same limit to the factors of the bracketed term. Since the cycle size $N$ is fixed, we can pass the limit inside the finite product
\[ \lim_{n\to\infty}\left( \prod_{j=1}^{N-1} \left(1 -\frac{j}{n}\right) \right)^{cn} = \prod_{j=1}^{N-1} \left( \lim_{n\to\infty}\left(1 -\frac{j}{n}\right)^n \right)^c = \prod_{j=1}^{N-1} e^{-jc}=e^{-c \sum_{j=1}^{N-1} j}=e^{-c\binom{N}{2}}. \]
Multiplying these limits together with the coefficient $\frac{1}{N}$ yields the result.
\end{proof}

\subsection{Eventually cyclic pairs and triples}
\label{Eventually_cyclic_pairs_and_triples}

In this section, we discuss special cases when $k=2,3$.

\begin{figure}
    \centering
\begin{tikzpicture}[scale=0.6,decoration={
    markings,
    mark=at position 0.35 with {\arrow{>}}}]

\begin{scope}[shift={(0,0)}]


    \draw[thick, postaction = {decorate}] (0,2) .. controls (0.5,1.25) and (3.5,2.25) .. (4,3);

    \node at (1.5,1.35) {$f$};

    \draw[thick, postaction = {decorate}] (4,3) .. controls (3.5,4.35) and (0.5,4.10) .. (0,4);

    \node at (1.5,4.45) {$g$};

    \draw[thick, postaction = {decorate}] (0,4) .. controls (-2,3.75) and (-3,-1) .. (0,-2);

    \draw[thick, fill] (0.25,2) arc (0: 360: 1.75 mm);

    \node at (-0.65,2) {$x_1$};

    \draw[thick, fill] (4.25,3) arc (0: 360: 1.75 mm);

    \node at (4.75,3) {$y_1$};

    \draw[thick] (-2.5,2.5) -- (-2.75,2.5) -- (-2.75,-2.5) -- (-2.5,-2.5);

    \node at (-3.5,0) {$N$};

\end{scope}

\begin{scope}[shift={(0,-2)}]


    \draw[thick, postaction = {decorate}] (0,2) .. controls (0.5,1.25) and (3.5,2.25) .. (4,3);

    \node at (1.5,1.35) {$f$};

    \draw[thick, fill] (0.25,2) arc (0: 360: 1.75 mm);

    \node at (-0.65,2) {$x_2$};

    \draw[thick, fill] (4.25,3) arc (0: 360: 1.75 mm);

    \node at (4.75,3) {$y_2$};

    \draw[thick, postaction = {decorate}] (4,3) .. controls (3.5,4.1) and (0.5,5) .. (0,4);

    \node at (1.5,4.75) {$g$};

\end{scope}

\begin{scope}[shift={(0,-4)}]

    \draw[thick, postaction = {decorate}] (0,2) .. controls (0.5,1.55) and (3.5,2.25) .. (4,3);

    \node at (1.5,1.47) {$f$};

    \draw[thick, postaction = {decorate}] (4,3) .. controls (3.5,4.1) and (0.5,5) .. (0,4);

    \node at (2.00,4.65) {$g$};

    \draw[thick, fill] (0.25,2) arc (0: 360: 1.75 mm);

    \node at (-0.65,1.75) {$x_3$};

    \draw[thick, fill] (4.25,3) arc (0: 360: 1.75 mm);

    \node at (4.75,3) {$y_3$};

\end{scope}


\begin{scope}[shift={(12,0)}]

\begin{scope}[shift={(0,2)}]

    \node at (3.00,3.45) {$g$};

    \draw[thick, postaction = {decorate}] (4,3) .. controls (3.5,3.35) and (0.5,3.10) .. (0,0);

    \draw[thick, fill] (4.25,3) arc (0: 360: 1.75 mm);

    \node at (4.75,3) {$y_1$};

\end{scope}

\begin{scope}[shift={(0,0)}]


    \draw[thick, postaction = {decorate}] (0,2) .. controls (0.5,1.25) and (3.5,2.25) .. (4,3);

    \node at (1.5,1.35) {$f$};

    \draw[thick, postaction = {decorate}] (4,3) .. controls (3.5,4.35) and (0.5,4.10) .. (0,4);

    \node at (-0.25,4.35) {$g$};

    \draw[thick, postaction = {decorate}] (0,4) .. controls (-2,3.75) and (-3,-1) .. (0,-2);

    \draw[thick, fill] (0.25,2) arc (0: 360: 1.75 mm);

    \node at (-0.65,2) {$x_1$};

    \draw[thick, fill] (4.25,3) arc (0: 360: 1.75 mm);

    \node at (4.75,3) {$y_2$};
        
\end{scope}

\begin{scope}[shift={(0,-2)}]


    \draw[thick, postaction = {decorate}] (0,2) .. controls (0.5,1.25) and (3.5,2.25) .. (4,3);

    \node at (1.5,1.35) {$f$};

    \draw[thick, fill] (0.25,2) arc (0: 360: 1.75 mm);

    \node at (-0.65,2) {$x_2$};

    \draw[thick, fill] (4.25,3) arc (0: 360: 1.75 mm);

    \node at (4.75,3) {$y_3$};

    \draw[thick, postaction = {decorate}] (4,3) .. controls (3.5,4.1) and (0.5,5) .. (0,4);

    \node at (1.5,4.75) {$g$};

\end{scope}

\begin{scope}[shift={(0,-4)}]


    \draw[thick, postaction = {decorate}] (0,2) .. controls (0.5,1.55) and (3.5,2.25) .. (4,3);

    \node at (1.5,1.50) {$f$};

    \draw[thick, postaction = {decorate}] (4,3) .. controls (3.5,4.1) and (0.5,5) .. (0,4);

    \node at (2.05,4.65) {$g$};

    \draw[thick, fill] (0.25,2) arc (0: 360: 1.75 mm);

    \node at (-0.65,1.75) {$x_3$};

    \draw[thick, fill] (4.25,3) arc (0: 360: 1.75 mm);

    \node at (4.75,3) {$y_4$};

\end{scope}

\begin{scope}[shift={(0,-6)}]


    \node at (1.5,1.60) {$f$};

    \draw[thick,postaction={decorate}] (0,2) .. controls (0.5,2) and (3,2) .. (4,3);

    \draw[thick, postaction = {decorate}] (4,3) .. controls (3.5,3.25) and (0.5,5) .. (0,6);

    \node at (3.5,3.70) {$g$};

    \draw[thick, fill] (0.25,2) arc (0: 360: 1.75 mm);

    \node at (-0.65,1.75) {$x_4$};

    \draw[thick, fill] (4.25,3) arc (0: 360: 1.75 mm);

    \node at (4.75,3) {$y_5$};

\end{scope}

\end{scope}

\end{tikzpicture}

    \caption{Left: example of a $3$-cyclic $2$-tuple.
    Right: example of one eventually $3$-cyclic $2$-tuple for $|X_1|=4$, $|X_2|=5$.}
    \label{fig_00015}
\end{figure}

\begin{example}
Suppose $k=2$ and that $X_1, X_2$ are disjoint sets of cardinality $n_1$ and $n_2$. An eventually $N$-cyclic pair $(f,g)$ consists of maps $f: X_1 \to X_2, \ g: X_2 \to X_1$, such that for some $m \geq 1$, the iterate $(gf)^m : X_1 \to X_1$ stabilizes to a subset $U_1 \subseteq X_1$ of size $N$, and $(gf)^m$ induces an $N$-cycle on $U_1$.

In this case, the digraph $\Gamma$ reduces to a directed bipartite graph between $X_1$ and $X_2$ containing a unique directed cycle of length $2N$. Removing the edges of this cycle yields a spanning in-forest of $\Gamma$ rooted along the cycle.

By Theorem~\ref{Weight of eventually cyclic functions} and Corollary~\ref{N_cyclic_rep_form}, the generating polynomial and the virtual character for eventually $N$-cyclic pairs, respectively, are
\begin{equation*}
\begin{split}
P_N(\mathbf{x}) 
&= \frac{(N!)^2}{N} \left(e_1(X_1)^{n_1-(N+1)} e_1(X_2)^{n_2-(N+1)}\right) \cdot  \\ 
&\hspace{4mm}\cdot \left(e_1(X_1)e_N(X_1)\, e_1(X_2)e_N(X_2) - (N+1)^2 e_{N+1}(X_1)e_{N+1}(X_2)\right) \\
&= (N!) (N-1)! \,
\chi\! \ \Big(
\left(V_1^{\otimes(n_2-(N+1))} \otimes V_2^{\otimes(n_1-(N+1))}\right)  \\ 
&\hspace{4mm}\otimes \left(
V_1^{\otimes N} \otimes V_2^{\otimes N}
- (N+1)^2 \Lambda^{N+1}V_1 \otimes \Lambda^{N+1}V_2
\right)\Big).
\end{split}
\end{equation*}

Applying Proposition \ref{prop_card_EC(N)}, the cardinality of eventually $N$-cyclic pairs is
\[
\#\EC(N) = \binom{n_1}{N}\binom{n_2}{N}\frac{(N!)^2}{N} \,
n_1^{\,n_1-(N+1)} n_2^{\,n_2-(N+1)} \left(n_1 n_2 - (n_1 - N)(n_2 - N)\right).
\]

See Figure~\ref{fig_00015} for an example where $k=2$ and $N=3$.
\end{example}

\begin{example}
Let $k=3$, with $X_1, X_2, X_3$ disjoint sets of cardinalities $n_1, n_2, n_3$. An \emph{eventually $N$-cyclic triple} $(f,g,h)$ consists of maps $f: X_1 \to X_2, \ g: X_2 \to X_3, \ h: X_3 \to X_1$, such that for some $m \geq 1$, the composition $(hgf)^m : X_1 \to X_1$ stabilizes to a subset $U_1 \subseteq X_1$ of size $N$, and induces an $N$-cycle on $U_1$.

In this case, the digraph $\Gamma$ has directed edges from $X_1$ to $X_2$, $X_2$ to $X_3$, and $X_3$ to $X_1$, and contains a unique directed cycle of length $3N$. Removing the edges of this cycle yields a spanning in-forest of $\Gamma$ rooted along the cycle.

By Theorem~\ref{Weight of eventually cyclic functions} and Corollary~\ref{N_cyclic_rep_form}, the generating polynomial and the virtual character
for eventually $N$-cyclic triples, respectively, are
\begin{equation*}
\begin{split}
P_N(\mathbf{x}) 
&= \frac{(N!)^3}{N} \left(\prod_{i=1}^3 e_1(X_i)^{n_i-(N+1)}\right)
\left(\prod_{j=1}^3 e_1(X_j)e_N(X_j) 
- (N+1)^3 \prod_{j = 1}^3 e_{N+1}(X_j)\right)  \\ 
&= (N!)^{2}(N - 1)! \,
\chi\! \ \Big(
\left(
V_1^{\otimes(n_3-(N+1))}
\otimes
V_2^{\otimes(n_1-(N+1))}
\otimes
V_3^{\otimes(n_2-(N+1))}
\right)   \\ 
&\hspace{4mm} \otimes\left(
V_1^{\otimes N} \otimes V_2^{\otimes N} \otimes V_3^{\otimes N}
- (N+1)^3 \Lambda^{N+1}V_1 \otimes \Lambda^{N+1}V_2 \otimes \Lambda^{N+1}V_3
\right)
\Big).
\end{split}
\end{equation*}

Applying Proposition \ref{prop_card_EC(N)}, the cardinality of eventually $N$-cyclic triples is
\[
\#\EC(N) 
= \prod_{i=1}^3 \binom{n_i}{N} \frac{(N!)^3}{N}
\ n_i^{\,n_i-(N+1)}
\left(
\prod_{j=1}^3 n_j - \prod_{j=1}^3 (n_j - N)
\right).
\]
\end{example}

\medskip

These formulas recover the eventually constant case when $N=1$.

\section{Multisymmetric polynomials associated to eventually \texorpdfstring{$\lambda$}{lambda}-cyclic functions}
\label{section_lambda_cyclic_functions}

Fix a positive integer $N\in\mathbb N$ and an integer partition $\lambda\vdash N$. For each number $l\leq N$, let $m_l$ denote the multiplicity of $l$ in $\lambda$.

A $k$-tuple $(f_1,\ldots,f_k)$ is \defn{eventually $\lambda$-cyclic} if iterated images of the composition $f_k\cdots f_1:X_1\to X_1$ eventually stabilize to a subset $(f_k\cdots f_1)^m(X_1)\subseteq X_1$ of cardinality $N$, on which $f_k\cdots f_1$ acts as a permutation of \defn{cycle type $\lambda$}. That is, for each $l\leq N$, $f_k\cdots f_1$ has $m_l$ disjoint cycles of length $l$. Let $\EC(\lambda)$ be the set of eventually $\lambda$-cyclic $k$-tuples, and define the generating function
\[
P_\lambda(\mathbf x):=\sum_{(f_1,\ldots,f_k)\in\EC(\lambda)}w(f_1,\ldots,f_k)=w(\EC(\lambda)).
\]

\subsection{Multisymmetric polynomials of eventually \texorpdfstring{$\lambda$}{lambda}-cyclic functions}
\label{lambda_cyclic_functions}

We will be brief since calculating $P_\lambda(\mathbf x)$ involves similar methods used in Sections~\ref{N-choice set decomposition} and \ref{subsection_wt_eventually_cyclic_fns}. We are still analyzing functions whose composition eventually stabilizes on sets of cardinality $N$, so we again decompose by $N$-choice sets $U$ and consider subsets of $\mathcal C_U$ and $\mathcal D_U$. Now, we simply take more general subsets of $\mathcal C_U$ and $\mathcal D_U$.

For each $N$-choice set $U$, let $\mathcal C_U^{(\lambda)}\subseteq \mathcal C_U$ be the subset of cyclic subgraphs $(f_1,\ldots,f_k)$ of $\Gamma$ on $U$ whose composition $f_k\cdots f_1:U_1\to U_1$ has cycle type $\lambda$. Also, define the corresponding subset $\mathcal D_U^{(\lambda)}\subseteq \mathcal D_U$ of spanning functional digraphs of $\Gamma$ whose cyclic part is in $\mathcal C_U^{(\lambda)}$. 

\begin{lemma}
    For each $N$-choice set $U$, the cardinality of $\mathcal C_U^{(\lambda)}$ is
    \[|\mathcal C_U^{(\lambda)}|=\frac{(N!)^k}{\prod_{l=1}^N l^{m_l}m_l!}.\]
\end{lemma}
\begin{proof}
    The subfamily $\mathcal C_U^{(\lambda)}$ corresponds to the $k$-tuples $(f_1,\ldots,f_k)$ of bijections $f_i:U_i\to U_{i+1}$ for which the composition $f_k\cdots f_1:U_1\to U_1$ has cycle type $\lambda$. It is known (see \cite{Sta12}) that the number of permutations of type $\lambda$ is
    \begin{equation}\label{cardinality_cycle_type_lambda}
        \frac{N!}{\prod_{l=1}^N l^{m_l}m_l!}.
    \end{equation}

    The remainder of the proof follows by the right-composition
    \[(f_1,\ldots,f_k)\mapsto (f_1,\ldots,f_{k-1},f_k\circ (f_{k-1}\cdots f_1)),\]
    as in Lemma \ref{Cycle enumeration lemma}. We simply replace the set of $N$-cycles on $g:U_1\to U_1$, of cardinality $(N-1)!$, with the set of permutations $g:U_1\to U_1$ of cycle type $\lambda$, of cardinality \eqref{cardinality_cycle_type_lambda}.
\end{proof}

Consider any spanning functional subgraph $D$ of $\Gamma$ with $k$-tuple $(f_1,\ldots,f_k)$ whose composition $f_k\cdots f_1:X_1\to X_1$ stabilizes at a subset $U_1\subseteq X_1$ of size $N:=|U_1|$. By the correspondence following Lemma \ref{Cycle bijection lemma}, any vertex $u\in X_1$ belongs to $V(C(D))$ if and only if $u$ belongs to the orbit of a cycle of $f_k\cdots f_1$, if and only if $u\in U_1$. Therefore, $V(C(D))\cap X_1=U_1$, and $U:=V(C(D))$ is the union of the iterated images:
\begin{equation}\label{N-choice set structure}
U_1,\,U_2:=f_1(U_1),\,U_3:=f_2(U_2),\,\ldots,\,U_k:=f_{k-1}(U_{k-1}).
\end{equation}

Note that $C(D)$ corresponds to the $k$-tuple of restrictions:
\[(f_1|_{U_1},f_2|_{U_2},\ldots,f_k|_{U_k}),\]
and each map $f_i:U_i\to U_{i+1}$ is a bijection, so $U$ is an $N$-choice set. Also, the composition of the restrictions $f_i|_{U_i}$ is equal to the composition of the full maps $f_i$ restricted to $U_1$:
\[f_k|_{U_k}\cdots f_1|_{U_1}=(f_k\cdots f_1)|_{U_1}.\]
Therefore, $D\in\mathcal D_U^{(\lambda)}$ if and only if the cyclic part of the composition $f_k\cdots f_1:X_1\to X_1$ has cycle type $\lambda$.

Thus, $\EC(\lambda)$ contains each subfamily $\mathcal D_U^{(\lambda)}$. Conversely, the composition $f_k\cdots f_1$ of each digraph $(f_1,\ldots,f_k)\in\EC(\lambda)$ stabilizes at some set $U_1\subseteq X_1$ of cardinality $N$. So, $(f_1,\ldots,f_k)\in\mathcal D_U^{(\lambda)}$, for $U$ the $N$-choice set given by \eqref{N-choice set structure}.

Since the classes $\mathcal D_U^{(\lambda)}$ partition $\EC(\lambda)$ as $U$ ranges over all $N$-choice sets, the generating function decomposes as 
\begin{equation}
\label{Plambda sum}
P_\lambda(\mathbf x)=\sum_U w(\mathcal D_U^{(\lambda)})=\frac{(N!)^k}{\prod_{l=1}^N l^{m_l}m_l!}\sum_U \left(\prod_{u\in U}x_u\right)w(\mathcal F_U).
\end{equation}

Observe that \eqref{Plambda sum} differs from \eqref{Sum over m-choice sets} by a constant factor, so we have the ratios
\begin{equation}
\label{Plambda pn ratio}
    P_\lambda(\mathbf x)=\frac{N}{\prod_{l=1}^N l^{m_l}m_l!} P_N(\mathbf x),\quad \#\EC(\lambda)=\frac{N}{\prod_{l=1}^N l^{m_l}m_l!} \#\EC(N).
\end{equation}

\begin{proposition}
\label{lambda_cyclic_generating_function}
The weighted enumeration of all eventually $\lambda$-cyclic $k$-tuples is given by 
\[
P_\lambda(\mathbf x) 
= \frac{(N!)^k}{\prod_{l=1}^N l^{m_l}m_l!} 
\prod_{i=1}^k e_1(X_i)^{n_{i-1}-(N+1)}
\left(
\prod_{j=1}^k e_1(X_j)e_N(X_j) - (N+1)^k\prod_{j=1}^k e_{N+1}(X_j)
\right),
\]
and the cardinality of the eventually $\lambda$-cyclic $k$-tuples is
    \[
    \#\EC(\lambda)=\frac{1}{\prod_{l=1}^N l^{m_l}m_l!}
\prod_{i=1}^k n_i^{\,n_{i-1}-(N+1)}\frac{n_i!}{(n_i-N)!}
\left(
\prod_{j=1}^k n_j - \prod_{j=1}^k (n_j - N)
\right).
\]
\end{proposition}
\begin{proof}
    Substitute the explicit formulas for $P_N(\mathbf x)$ and $\#\EC(N)$ from Theorem \ref{Weight of eventually cyclic functions} and Proposition \ref{prop_card_EC(N)} into \eqref{Plambda pn ratio}.
\end{proof}

\begin{remark}
    This formula can again be expressed as a virtual character of $\prod_{i=1}^k \GL(n_i)$, only differing from Corollary~\ref{N_cyclic_rep_form} by a constant as given in \eqref{Plambda pn ratio}.
\end{remark}

\subsection{Eventually \texorpdfstring{$\lambda$}{lambda}-cyclic triples}
\label{subsection_eventually_lambda_cyclic_triple}

\begin{example}
\label{prop_m_n_2}
Let $X_1,X_2, X_3$ have cardinality $m, n,2$, respectively. 
Let $f:X_1\rightarrow X_2$, $g: X_2\rightarrow X_3$, $h:X_3\rightarrow X_1$.
The number of eventually constant maps is
\begin{equation}
\label{eqn_event_constant_m1_m2_2}
\begin{split}
|(m, n, 2)| &= 2^n n^m m^2 - 
\sum_{k=1}^{m-1} \sum_{\ell=1}^{n-1} \binom{m}{k} \binom{n}{\ell} 
|(k,\ell,1)|  |(m-k,n-\ell,1)|  \\ 
&\qquad 
 - \sum_{k=1}^{m-1} \sum_{\ell=1}^{n-1} 
\binom{m}{k} \binom{n}{\ell} k \ell^k \cdot 1 (m-k)(n-\ell)^{m-k} \cdot 1.
\end{split}
\end{equation}

The expression $2^n n^m m^2$ provides the total number of admissible maps. For the second summation, we are counting the number of maps which contain two $3$-cycles. See Figure~\ref{fig_00008}, left. So fix $f,g,h$. We partition $X_1 \cup X_2 \cup X_3$ into two disjoint sets: all the elements and maps that map into one $3$-cycle (say these elements are boxed) and the rest that map into the other $3$-cycle (these elements are circled). These two sets are disjoint since each proper subset contains a unique fixed point. Let $1\leq k < m$ and $1\leq \ell < n$. There are $k$ elements of $X_1$ to choose from (to be a red solid box), $\ell$ elements of $X_2$ to choose from (to be a blue dashed box), and these $\ell$ elements all map to the boxed element in $X_3$ to obtain $|(k,\ell,1)|$. Similarly, we have  $m-k$ elements to choose from $X_1$ to be circled and $n-\ell$ elements to choose from $X_2$ to be circled, and these $n-\ell$ elements map to the other (circled) element in $X_3$ to obtain $|(m-k,n-\ell,1)|$. Since there are $\binom{m}{k}$ choices for the $k$ element subset of boxes in $X_1$ and $\binom{n}{\ell}$ choices for the $\ell$ element subset of boxes in $X_2$, we now sum over all $k$ and $\ell$ to obtain the second summation. 

For the last summation, we are counting all maps $(f,g,h)$ such that it contains a $6$-cycle. 
See Figure~\ref{fig_00008}, right.
Let $1\leq k< m$ and $1\leq \ell <n$.
There are $k$ choices to send $x_{31}$ to a red solid box. Then there are $\ell^k$ choices to send a red solid box to a blue dashed box, and $1$ choice to send the blue dashed boxes to the orange dotted box. Then there are $m-k$ choices to send $x_{32}$ to a red solid circle, $(n-\ell)^{m-k}$ choices to send a red solid circle to a blue dashed circle, and $1$ choice to send the blue dashed circles to the orange dotted circle. Finally, there are $\binom{m}{k}$ ways to choose $k$ subsets of the $m$ set to be a red solid box and $\binom{n}{\ell}$ ways to choose $\ell$ subset of the $n$ set to be blue dashed boxes. We sum over all $k$ and $\ell$ to obtain the last term.
\end{example}

\begin{figure}
    \centering
\begin{tikzpicture}[scale=0.5,decoration={
    markings,
    mark=at position 0.5 with {\arrow{>}}}]


\begin{scope}[shift={(0,0)}]

\node at (0.0,7.10) {$X_1$};

\node at (4.0,7.10) {$X_2$};

\node at (8.60,7.10) {$X_3$};

\node at (0,6) {$x_{11}$};

\node at (0,4.5) {$x_{12}$};

\node at (0,3.2) {$\vdots$};

\node at (0,1.6) {$\vdots$};

\node at (0,-0.1) {$x_{1m}$};

\node at (4,6) {$x_{21}$};

\node at (4,4.5) {$x_{22}$};

\node at (4,3.2) {$\vdots$};

\node at (4,1.5) {$x_{2n}$};

\node at (8,6) {$x_{31}$};

\node at (8,4.0) {$x_{32}$};

\draw[line width=0.04cm,red] (0,6) ellipse (0.85 and 0.6);

\draw[line width=0.04cm,red] (0,3.0) ellipse (0.85 and 0.6);

\draw[line width=0.04cm,red] (0,-0.1) ellipse (0.85 and 0.6);

\draw[line width=0.04cm,red] (-0.75,5) rectangle (0.75,4);

\draw[line width=0.04cm,red] (-0.75,1.9) rectangle (0.75,0.9);

\draw[line width=0.05cm,blue,dashed] (3.25,6.5) rectangle (4.75,5.5);

\draw[line width=0.05cm,blue,dashed] (4,4.5) ellipse (0.85 and 0.6);

\draw[line width=0.05cm,blue,dashed] (3.25,3.50) rectangle (4.75,2.50);

\draw[line width=0.05cm,blue,dashed] (4,1.5) ellipse (0.85 and 0.6);

\draw[line width=0.06cm,orange,dash dot] (8,6) ellipse (0.85 and 0.6);

\draw[line width=0.06cm,orange,dash dot] (7.25,4.5) rectangle (8.75,3.5);

\draw[thick,->] (1,6.0) -- (3,4.7);

\draw[thick,->] (1,4.5) -- (3,6.0);

\draw[thick,->] (1,3.0) -- (3,4.3);

\draw[thick,->] (1,1.5) -- (3,3.00);

\draw[thick,->] (1,0.0) -- (3,1.40);

\draw[thick,->] (5,6.0) -- (7,4.20);

\draw[thick,->] (5,4.5) -- (7,6.250);

\draw[thick,->] (5,3.0) -- (7,3.80);

\draw[thick,->] (5,1.5) -- (7.1,5.65);

\draw[thick] (0,7.75) .. controls (2,8) and (6,8) .. (8,6.75);

\draw[thick,->] (0,7.75) .. controls (-3,7.5) and (-3,4) .. (-1,3);

\draw[thick] (0,-1.5) .. controls (2,-1.5) and (6,-1) .. (8,3.35);

\draw[thick,<-] (-1,1.4) .. controls (-2,1) and (-2,-1.5) .. (0,-1.5);

\end{scope}


\begin{scope}[shift={(15,0)}]


\node at (0.0,7.10) {$X_1$};

\node at (4.0,7.10) {$X_2$};

\node at (8.60,7.10) {$X_3$};

\node at (0,6) {$x_{11}$};

\node at (0,4.5) {$x_{12}$};

\node at (0,3.2) {$\vdots$};

\node at (0,1.6) {$\vdots$};

\node at (0,-0.1) {$x_{1m}$};

\node at (4,6) {$x_{21}$};

\node at (4,4.5) {$x_{22}$};

\node at (4,3.2) {$\vdots$};

\node at (4,1.6) {$\vdots$};

\node at (4,-0.1) {$x_{2n}$};

\node at (8,6) {$x_{31}$};

\node at (8,4.0) {$x_{32}$};

\draw[line width=0.04cm,red] (0,6) ellipse (0.85 and 0.6);

\draw[line width=0.04cm,red] (0,3.0) ellipse (0.85 and 0.6);

\draw[line width=0.04cm,red] (0,1.4) ellipse (0.85 and 0.6);

\draw[line width=0.04cm,red] (-0.75,5) rectangle (0.75,4);

\draw[line width=0.04cm,red] (-0.75,0.4) rectangle (0.75,-0.6);

\draw[line width=0.05cm,blue,dashed] (3.25,6.5) rectangle (4.75,5.5);

\draw[line width=0.05cm,blue,dashed] (4,4.5) ellipse (0.85 and 0.6);

\draw[line width=0.05cm,blue,dashed] (3.25,3.50) rectangle (4.75,2.50);

\draw[line width=0.05cm,blue,dashed] (4,1.4) ellipse (0.85 and 0.6);

\draw[line width=0.05cm,blue,dashed] (3.25,0.40) rectangle (4.75,-0.60);

\draw[line width=0.06cm,orange,dash dot] (8,6) ellipse (0.85 and 0.6);

\draw[line width=0.06cm,orange,dash dot] (7.25,4.5) rectangle (8.75,3.5);

\draw[thick,->] (1,6.0) -- (3,4.7);

\draw[thick,->] (1,4.5) -- (3,6.0);

\draw[thick,->] (1,3.0) -- (3,4.3);

\draw[thick,->] (1,1.5) -- (3,1.50);

\draw[thick,->] (1,0.0) -- (3,0.0);

\draw[thick,->] (5,6.0) -- (7,4.20);

\draw[thick,->] (5,4.5) -- (7,6.250);

\draw[thick,->] (5,3.0) -- (7,3.80);

\draw[thick,->] (5,1.5) -- (7.1,5.65);

\draw[thick,->] (5,0.0) -- (7.1,3.25);

\draw[thick] (0,7.75) .. controls (2,8) and (6,8) .. (8,6.75);

\draw[thick,->] (0,7.75) .. controls (-2.5,7.5) and (-2.5,5) .. (-1,4.5);

\draw[thick] (0,-1.5) .. controls (3,-1.5) and (6.75,-2.25) .. (8,3.35);

\draw[thick,<-] (-1,1.4) .. controls (-2,1) and (-2,-1.5) .. (0,-1.5);

\end{scope}


\end{tikzpicture}
    \caption{Left: Example of a map with two $3$-cycles for sets of sizes $(m,n,2)$. We partition $X_1 \cup X_2 \cup X_3$ into two disjoint sets, one set corresponding to one $3$-cycle (they are circled above) and another set corresponding to the other $3$-cycle (they are boxed above). Right: Example of a map with one $6$-cycle for sets of sizes $(m,n,2)$.}
    \label{fig_00008}
\end{figure}


\bibliographystyle{amsalpha} 
\bibliography{nilpotent_finite}

@article {Lei21,
    AUTHOR = {Leinster, Tom},
     TITLE = {The probability that an operator is nilpotent},
   JOURNAL = {Amer. Math. Monthly},
  FJOURNAL = {American Mathematical Monthly},
    VOLUME = {128},
      YEAR = {2021},
    NUMBER = {4},
     PAGES = {371--375},
}

@article {CIKLR25,
    AUTHOR = {Chen, Weixi and Im, Mee Seong and Khovanov, Mikhail and
              Lillja, Catherine and Rugo, Nicolas},
     TITLE = {Pairs of eventually constant maps and nilpotent pairs},
   JOURNAL = {Lett. Math. Phys.},
  FJOURNAL = {Letters in Mathematical Physics},
    VOLUME = {116},
      YEAR = {2026},
    NUMBER = {4},
     PAGES = {Paper No. 77},
}

@inbook{Cayley09, 
    AUTHOR = {Cayley, Arthur}, 
     TITLE = {A theorem on trees}, 
     PLACE = {Cambridge}, 
    SERIES = {Cambridge Library Collection - Mathematics}, 
 BOOKTITLE = {The Collected Mathematical Papers}, 
 PUBLISHER = {Cambridge University Press}, 
      YEAR = {2009}, 
     PAGES = {26--28}, 
COLLECTION = {Cambridge Library Collection - Mathematics}
}

@article {GHI26_set_quiver,
    AUTHOR = {Green, Radford and Holmes, Cornell and Im, Mee Seong},
     TITLE = {Multisymmetric polynomials on set-theoretic quiver representations},
   JOURNAL = {arXiv preprint \href{https://arxiv.org/abs/2606.16956}{arXiv:2606.16956}},
      YEAR = {2026},
     PAGES = {1--35},
}

@article {CILR25,
    AUTHOR = {Chen, Weixi and Im, Mee Seong and Lillja, Catherine and Rugo, Nicolas},
     TITLE = {Eventually constant maps for two sets and nilpotent pairs},
   JOURNAL = {arXiv preprint \href{https://arxiv.org/abs/2512.05269}{arXiv:2512.05269}, to appear in Contemp. Math.},
      YEAR = {2025},
     PAGES = {1--21},
}

@book {Sta24,
    AUTHOR = {Stanley, Richard P.},
     TITLE = {Enumerative {C}ombinatorics. {V}ol. 2},
    SERIES = {Cambridge Studies in Advanced Mathematics},
    VOLUME = {208},
   EDITION = {Second},
      NOTE = {With an appendix by Sergey Fomin},
 PUBLISHER = {Cambridge University Press, Cambridge},
      YEAR = {2024},
     PAGES = {xvi+783},
}

@book {GKP89,
    AUTHOR = {Graham, Ronald L. and Knuth, Donald E. and Patashnik, Oren},
     TITLE = {Concrete mathematics},
      NOTE = {A foundation for computer science},
 PUBLISHER = {Addison-Wesley Publishing Company, Advanced Book Program,
              Reading, MA},
      YEAR = {1989},
     PAGES = {xiv+625},
}

@book {Die25,
    AUTHOR = {Diestel, Reinhard},
     TITLE = {Graph theory},
    SERIES = {Graduate Texts in Mathematics},
    VOLUME = {173},
   EDITION = {Sixth},
 PUBLISHER = {Springer, Berlin},
      YEAR = {2025},
     PAGES = {xx+454},
}

@book {Wes96,
    AUTHOR = {West, Douglas B.},
     TITLE = {Introduction to graph theory},
 PUBLISHER = {Prentice Hall, Inc., Upper Saddle River, NJ},
      YEAR = {1996},
     PAGES = {xvi+512},
}

@book {HJ13,
    AUTHOR = {Horn, Roger A. and Johnson, Charles R.},
     TITLE = {Matrix analysis},
   EDITION = {Second},
 PUBLISHER = {Cambridge University Press, Cambridge},
      YEAR = {2013},
     PAGES = {xviii+643},
}

@book {GR01,
    AUTHOR = {Godsil, Chris and Royle, Gordon},
     TITLE = {Algebraic graph theory},
    SERIES = {Graduate Texts in Mathematics},
    VOLUME = {207},
 PUBLISHER = {Springer-Verlag, New York},
      YEAR = {2001},
     PAGES = {xx+439},
}

@book {HJ94,
    AUTHOR = {Horn, Roger A. and Johnson, Charles R.},
     TITLE = {Topics in matrix analysis},
      NOTE = {Corrected reprint of the 1991 original},
 PUBLISHER = {Cambridge University Press, Cambridge},
      YEAR = {1994},
     PAGES = {viii+607},
}

@article {FH58,
    AUTHOR = {Fine, Nathan J. and Herstein, Israel N.},
     TITLE = {The probability that a matrix be nilpotent},
   JOURNAL = {Illinois J. Math.},
  FJOURNAL = {Illinois Journal of Mathematics},
    VOLUME = {2},
      YEAR = {1958},
     PAGES = {499--504},
}

@article{Hai10,
      AUTHOR = {Haiman, Mark},  
      TITLE = {Notes on the {M}atrix-{T}ree theorem and {C}ayley's tree enumerator}, 
      JOURNAL = {Combinatorics \href{https://math.berkeley.edu/~mhaiman/math172-spring10/matrixtree.pdf}{https://math.berkeley.edu/$\sim$mhaiman/math172-spring10/matrixtree.pdf}},
      YEAR = {2010},
      PAGES = {1--7},
}

@article {FS58_eng,
    AUTHOR = {Fiedler, Miroslav and Sedl\'{a}\v{c}ek, Ji\v{r}\'i},
     TITLE = {On W-bases of directed graphs ({\"{U}}ber {W}urzelbasen von gerichteten {G}raphen)},
   JOURNAL = {\v{C}asopis P\v{e}st. Mat.},
  FJOURNAL = {\v{C}eskoslovensk\'a{} Akademie V\v{e}d. \v{C}asopis Pro P\v{e}stov\'an\'i\ Matematiky},
    VOLUME = {83},
      YEAR = {1958},
     PAGES = {214-225},
}

@article {Sta11,
    AUTHOR = {Stacey, Andrew},
     TITLE = {Comparative smootheology},
   JOURNAL = {Theory Appl. Categ.},
  FJOURNAL = {Theory and Applications of Categories},
    VOLUME = {25},
      YEAR = {2011},
     PAGES = {No. 4, 64--117},
}

@book {Bol98,
    AUTHOR = {Bollob\'as, B\'ela},
     TITLE = {Modern {G}raph {T}heory},
    SERIES = {Graduate Texts in Mathematics},
    VOLUME = {184},
 PUBLISHER = {Springer-Verlag, New York},
      YEAR = {1998},
     PAGES = {xiv+394},
}

@article {KR,
    AUTHOR = {Khovanov, Mikhail and Robert, Louis-Hadrien},
     TITLE = {Foam evaluation and {K}ronheimer-{M}rowka theories},
   JOURNAL = {Adv. Math.},
  FJOURNAL = {Advances in Mathematics},
    VOLUME = {376},
      YEAR = {2021},
     PAGES = {Paper No. 107433, 59},
}

@article {CK78,
    AUTHOR = {Chaiken, Seth and Kleitman, Daniel J.},
     TITLE = {Matrix tree theorems},
   JOURNAL = {J. Combinatorial Theory Ser. A},
  FJOURNAL = {Journal of Combinatorial Theory. Series A},
    VOLUME = {24},
      YEAR = {1978},
    NUMBER = {3},
     PAGES = {377--381},
      ISSN = {0097-3165},
}

@article {PakPostnikov94,
    AUTHOR = {Pak, Igor and Postnikov, Alexander},
     TITLE = {Enumeration of spanning trees of graphs},
   JOURNAL = {preprint},
      YEAR = {1994},
     PAGES = {1--18},
}

@book {FS09,
    AUTHOR = {Flajolet, Philippe and Sedgewick, Robert},
     TITLE = {Analytic combinatorics},
 PUBLISHER = {Cambridge University Press, Cambridge},
      YEAR = {2009},
     PAGES = {129},
}

@book {Sta12,
    AUTHOR = {Stanley, Richard P.},
     TITLE = {Enumerative combinatorics. {V}olume 1},
    SERIES = {Cambridge Studies in Advanced Mathematics},
    VOLUME = {49},
   EDITION = {Second},
 PUBLISHER = {Cambridge University Press, Cambridge},
      YEAR = {2012},
     PAGES = {xiv+626},
}

\end{document}